\documentclass[final]{elsarticle}

\usepackage{hyperref}

\hypersetup{pdfauthor={Matthias Lenz},pdftitle={Zonotopal algebra and forward exchange matroids}, 
  pdfsubject={Zonotopal Algebra}}
 
 \usepackage[T1,T2A]{fontenc}
 \usepackage[ansinew]{inputenc}
\usepackage{amssymb} 
\usepackage{amsmath} 
\usepackage{amsfonts} 
\usepackage{amsthm}  
\usepackage{upref} 
\usepackage{graphicx}
\usepackage{color}
\usepackage{url}
\usepackage{paralist} 
\usepackage{caption}

\usepackage[russian,british]{babel}
\usepackage{mathtext} 
\DeclareSymbolFont{T2Aletters}{T2A}{cmr}{m}{it} 

\usepackage{calc}
\usepackage{bm} 

\usepackage[all]{xy}

\usepackage{fancyhdr}
\newtheorem{Theorem}{Theorem}[section]
\newtheorem{Corollary}[Theorem]{Corollary}
\newtheorem{Proposition}[Theorem]{Proposition}
\newtheorem{Lemma}[Theorem]{Lemma}

\theoremstyle{definition}
\newtheorem{Definition}[Theorem]{Definition}

\newtheorem{Question}[Theorem]{Question}
\newtheorem{Example}[Theorem]{Example}

\theoremstyle{remark}
\newtheorem{Remark}[Theorem]{Remark}

\newcommand{\di}{\mathrm{d}} 
\newcommand{\DM}{\mathop{\mathrm{DM}}\nolimits}
\newcommand{\supp}{\mathop{\mathrm{supp}}\nolimits}

\newcommand{\vol}{\mathop{\mathrm{vol}}}
\newcommand{\abs}[1]{\left|#1\right|}
\newcommand{\rank}{\mathop{\mathrm{rk}}}

\newcommand{\ann}{o} 

\newcommand{\st}{s.\,t.\ } 
\newcommand{\ie}{\textit{i.\,e.\ }} 
\newcommand{\eg}{\textit{e.\,g.\ }} 

\newcommand{\N}{\mathbb{N}}
\newcommand{\Z}{\mathbb{Z}}

\newcommand{\R}{\mathbb{R}}
\newcommand{\K}{\mathbb{K}}

\newcommand{\Bcal}{\mathcal{B}}

\newcommand{\Dcal}{\mathcal{D}}
\newcommand{\Hcal}{\mathcal{H}}
\newcommand{\Ical}{\mathcal{I}}

\newcommand{\Pcal}{\mathcal{P}}

\newcommand{\Jcal}{\mathcal{J}}

\newcommand{\Bcyr}{\text{\textit{\CYRB}}} 

\newcommand{\ideal}{\mathop{\mathrm{ideal}}}
\newcommand{\spa}{\mathop{\mathrm{span}}}
\newcommand{\cone}{\mathop{\mathrm{cone}}}
\newcommand{\BB}{ \mathbb B}

\newcommand{\Lcal}{\mathcal L}

\newcommand{\diff}[1]{\frac{\partial}{\partial #1}}
\newcommand{\ex}{\mathop{\mathrm{ex}}}
\newcommand{\hilb}{\mathop{\mathrm{Hilb}}}
\newcommand{\clos}{\mathop{\mathrm{cl}}\nolimits} 
\newcommand{\sym}{\mathop{\mathrm{Sym}}\nolimits} 
\newcommand{\pair}[2]{\langle #1,#2 \rangle}

\numberwithin{equation}{section}

\newcommand{\Xintro}{%
Let $X\subseteq U\cong \R^r$ be a finite list of vectors that spans $U$.
}
\newcommand{\XintroCard}{%
Let $X\subseteq U\cong \R^r$ be a  list of $N$ vectors that spans $U$.
}

\newcommand{\XBBintro}{%
 Let $X\subseteq U\cong \R^r$ be a finite list of vectors that spans $U$ and let 
 $\BB'$ be an arbitrary subset of its  set of bases $\BB(X)$.
}

\newcommand{\XBBintroFEP}{%
Let $X\subseteq U\cong \R^r$ be a finite list of vectors that spans $U$ 
 and let $\BB' \subseteq \BB(X)$ be a set of bases with the forward exchange
 property.
}

\newcommand{\XBBintroFEPcard}{%
Let $X\subseteq U\cong \R^r$ be a list of $N$ vectors that spans $U$ 
 and let $\BB' \subseteq \BB(X)$ be a set of bases with the forward exchange
 property.
}

\date{\today}

\begin{document}
\title{Zonotopal algebra and forward exchange matroids}
\author{Matthias Lenz\fnref{fn1,fn2}}
\ead{matthias.lenz@unifr.ch}
%
%
\address{%
Technische Universit\"at Berlin \\
Sekretariat MA 4-2 \\
Stra\ss e des 17.~Juni 136 \\
10623 Berlin \\
GERMANY
}
\fntext[fn1]{Present address: D\'epartement de math\'ematiques,
Universit\'e de Fribourg,
Chemin du Mus\'ee 23,
1700 Fribourg, Switzerland}
\fntext[fn2]{The author was supported by
a Sofia Kovalevskaya Research Prize of Alexander von Humboldt Foundation awarded to Olga Holtz
and by a PhD scholarship from the Berlin Mathematical
School (BMS)}


\begin{abstract}
%
%
Zonotopal algebra is the study of a family of pairs of dual vector spaces of 
multivariate polynomials that can be associated with a list of vectors $X$. 
It connects objects from combinatorics, geometry, and approximation theory.
The origin of zonotopal algebra is the pair $(\Dcal(X),\Pcal(X))$, where   
 $\Dcal(X)$ denotes the Dahmen--Micchelli space that is spanned by the local pieces of the box spline and 
 $\Pcal(X)$ is a space spanned by products of linear forms.

The first main result of this paper is the construction of a canonical basis for $\Dcal(X)$. We show that it is dual 
 to the canonical basis for $\Pcal(X)$ that is already known.

The second main result of this paper is the construction of a new family of zonotopal spaces that is far more general
 than the ones that were recently studied by Ardila--Postnikov, Holtz--Ron, Holtz--Ron--Xu, Li--Ron, and others.
 We call the underlying combinatorial structure of those spaces forward exchange matroid. 
 A forward exchange matroid is an ordered matroid
together with a subset of its set of bases that satisfies a weak version of the
 basis exchange axiom.
\end{abstract}
\maketitle
 
\thispagestyle{fancy}

\section{Introduction}
A finite list of vectors $X$ gives rise to a large number of objects in
   combinatorics, algebraic and discrete geometry, commutative algebra, and approximation theory.
Examples include matroids, hyperplane arrangements and zonotopes, fat point ideals, and box splines.
 In the 1980s, various authors in the approximation theory  community started studying
algebraic structures that capture information about 
 splines 
(\eg
\cite{akopyan-saakyan-1988,boor-hoellig-1982, dyn-ron-1990}).
One important example is the Dahmen--Micchelli space $\Dcal(X)$, that is 
 spanned by the local pieces of the box spline and their partial derivatives.
See \cite[Section 1.2]{holtz-ron-2011} for a historic survey and the book \cite{BoxSplineBook} 
for a treatment of polynomial spaces appearing in the theory of box splines.
 Related results were obtained independently by authors interested in hyperplane arrangements
(\eg \cite{orlik-terao-1994}).

The space $\Pcal(X)$ that is dual to $\Dcal(X)$ was introduced in \cite{akopyan-saakyan-1988,dyn-ron-1990}.
It is spanned by products of linear forms and it can be written as the
Macaulay inverse system (or kernel) of an ideal generated
 by powers of linear forms \cite{boor-dyn-ron-1991}.
Ideals of this type and their inverse systems 
are also studied in the literature on fat point ideals
 \cite{emsalem-iarrobino-1995,geramita-schenck-1998} and graph orientations \cite{backman-hopkins-2015}.

In addition to the aforementioned pair of spaces $(\Dcal(X),\Pcal(X))$,
 Olga Holtz and Amos Ron introduced two more pairs of spaces with interesting combinatorial properties
\cite{holtz-ron-2011}.
 They named the theory of those spaces  \emph{Zonotopal Algebra}.
 This name reflects the fact that there are various connections between zonotopal spaces
 and the lattice points in the zonotope defined by $X$
 if the list $X$ is 
 unimodular.

Subsequently, those results were further generalised by
Olga Holtz, Amos Ron, and 
 Zhiqiang Xu \cite{holtz-ron-xu-2012} as well as
  Nan Li and Amos Ron \cite{li-ron-2013}.
 Federico Ardila and 
 Alex Postnikov  studied generalised $\Pcal$-spaces and connections with power ideals \cite{ardila-postnikov-2009}.
 Bernd Sturmfels and Zhiqiang Xu established a connection with Cox rings \cite{sturmfels-xu-2010}.
 Further work on spaces of $\Pcal$-type includes \cite{berget-2010,brion-vergne-1999,lenz-hzpi-2012,lenz-arithmetic-2016,wagner-1999}.

 Zonotopal algebra is closely related to matroid theory: the Hilbert series of zonotopal spaces
  only depend on the matroid structure of the list $X$.

It is  known that there is a canonical way to construct bases for the spaces of
 $\Pcal$-type  \cite{ardila-postnikov-2009,dyn-ron-1990, holtz-ron-2011,lenz-todd-2015,li-ron-2013}. 
 The first of the two main results in this paper is  an algorithm
 that constructs a basis for spaces of $\Dcal$-type.
 Two different 
 algorithms  are already known  \cite{dahmen-1990,deBoor-Ron-computational-1992}. 
 However, our algorithm has several advantages over the other two: 
 it is canonical 
 and it yields a basis that is dual to
 the known basis  for the $\Pcal$-space.
 Here, canonical means that the basis that we obtain only depends on the order of the elements in the list $X$ 
 and not on any further choices.
 In \cite{lenz-todd-2015}, the author used the duality between the bases of the  $\Pcal$-space  and the $\Dcal$-space  to prove a slight generalisation 
 of the Khovanskii--Pukhlikov formula \cite{pukhlikov-khovanski-1992} that  relates the volume  and the number of lattice points in a smooth 
 lattice polytope.

 Our second main result is that far more general pairs of zonotopal spaces with nice properties can be constructed than 
 the ones that were previously known.
 We define a new combinatorial structure called forward exchange matroid. 
A forward exchange matroid is an ordered matroid
together with a subset of its set of bases that satisfies a weak version of the
 basis exchange axiom. %
 This is the underlying structure of the generalised zonotopal $\Dcal$-spaces and $\Pcal$-spaces that we introduce. 
 Potentially related structures were very recently studied by Jos\'e Samper in a combinatorial context \cite{samper-fpsac-toappear-2016}.

\smallskip
All objects mentioned so far are part of what we call the continuous theory.
If the list $X$ is contained in a lattice (\eg  $\Z^d$), an even wider spectrum of mathematical objects appears. 
We call this the
 discrete theory. Every object in the continuous theory has a discrete analogue: 
vector partition functions correspond to box splines and
 toric arrangements 
 correspond to hyperplane arrangements.
The local pieces of the vector partition function are quasi-polynomials that span the %
discrete Dahmen--Micchelli
 space $\DM(X)$.
Both theories are nicely explained in the recent book
by
Corrado De~Concini and Claudio Procesi \cite{concini-procesi-book}.
The combinatorics of the discrete case is 
 captured by arithmetic matroids 
which were  recently introduced by Luca Moci and Michele D'Adderio \cite{moci-adderio-2013,moci-tutte-2012}.

Vector partition functions arise for example in representation theory as 
Kostant partition function, when the list $X$ is chosen to be the set of positive roots of a simple Lie algebra
(\eg \cite{cochet-2005}).
In a series of articles, 
Corrado De~Concini, Claudio Procesi, and Mich\`ele Vergne, as well as 
Francesco Cavazzani and Luca Moci
have studied  applications to equivariant
index theory of elliptic operators
and showed that some of these spaces can be
``geometrically realised'' as   equivariant cohomology or  $K$-theory of certain differentiable manifolds
\cite{cavazzani-moci-2016,deconcini-procesi-vergne-2010a,deconcini-procesi-vergne-2010b,deconcini-procesi-vergne-2011,deconcini-procesi-vergne-2010c,deconcini-procesi-vergne-infinitesimal-2013}.

This paper deals only with the continuous theory. However, the two theories overlap if the
 list of vectors $X$ is 
 unimodular. We hope that %
 the results in this paper
 can be transferred to the discrete case in the future.

\subsection{Notation}
\label{Subsection:Notation}
Our basic object of study is a list of vectors $X=(x_1,\ldots, x_N)$ that span an $r$-dimensional
 real vector space $U \cong \R^r$. 
 The dual space $U^*$ is denoted by $V$.
We slightly abuse notation by using the symbol $\subseteq$ for sublists. 
For $Y\subseteq X$, $X\setminus Y$ denotes the deletion of a sublist, \ie $(x_1,x_2)\setminus (x_1)=(x_2)$
 even if $x_1=x_2$.  
The list $X$ comes with a natural ordering: we say that 
 $x_i < x_j$ if and only if $i<j$.

Note that $X$ can be identified with a linear  map $\R^N \to U$ and after the choice of a basis with  
 an $(r \times N)$-matrix with real entries.

We will consider families of pairs of dual spaces $(\Dcal(X,\cdot),\Pcal(X,\cdot))$.
 The space $\Dcal(X,\cdot)$ is contained in $\sym(V)$, the symmetric algebra over $V$ and $\Pcal(X,\cdot)$ is contained
 in $\sym(U)$. 
 The \emph{symmetric algebra} is a base-free version of the ring of polynomials over a vector space. 
We fix a basis $(s_1,\ldots, s_r)$  for $U$ and $(t_1,\ldots, t_r)$ denotes the dual basis for $V$,
 \ie  $t_i(s_j)=\delta_{ij}$, where $\delta_{ij}$ denotes the Kronecker delta.
 The choice of the 
 basis  determines isomorphisms
 $\sym(U)\cong \R[s_1,\ldots, s_r]$ and $\sym(V)\cong   \R[t_1,\ldots, t_r]$.
For $x\in U$,  $x^\circ\subseteq V$ denotes the \emph{annihilator} of $x$, \ie 
 $x^\ann := \{ f\in V : f(x) = 0  \}$.
For more background on algebra, see \cite{concini-procesi-book} or \cite{eisenbud-1995}.

As usual, $\chi_S : S\to \{0,1\}$ denotes the indicator function of a set $S$.

\subsection{Matroids}
An \emph{ordered matroid} on $N$ elements is a pair
$(A,\BB)$ where 
$A$ is an (ordered) list with $N$ elements and $\BB$ is a non-empty set
 of sublists of $A$ that satisfies
the following axiom:
\begin{equation}
\label{equation:ExchangeProperty}
\begin{split}
 &\text{Let } B, B'\in \BB \text{ and } b\in B\setminus B'. \\
 &\text{Then, there exists $b'\in B'\setminus B$ \st } 
  (B\setminus b)\cup b' \in \BB.
\end{split}
  \end{equation}
$\BB$ is called the \emph{set of bases} of the matroid $(A,\BB)$.
One can easily show that all elements of $\BB$ have the same cardinality.
This number $r$ is called the \emph{rank} of the matroid $(A,\BB)$. %
A set $I\subseteq A$ is called \emph{independent} if it is a subset of a basis.
The \emph{rank} of $Y\subseteq A$ is defined as the cardinality of a maximal  independent set 
contained in $Y$. It is denoted $\rank(Y)$. 
The \emph{closure} of $Y$  is defined as $\clos(Y):=\{ x\in A : \rank(Y \cup x)=\rank(Y) \}$.
A set $C\subseteq X$ is called a \emph{flat} if $C=\clos(C)$.

A set $C\subseteq X$ is called a
\emph{cocircuit} if $C\cap B\neq \emptyset$ for all bases $B\in \BB$ and $C$ is minimal with this property.
Cocircuits of cardinality one are called \emph{coloops}. An element that is not contained in 
any basis is called a \emph{loop}. 

Now fix a basis $B\in \BB$. 
An element $b\in B$ is called \emph{internally active} in $B$ if $b=\max (A\setminus \clos(B\setminus b))$,
\ie $b$ is the maximal element of the unique cocircuit %
contained in $(A\setminus B)\cup b$.
The set of internally active  elements in $B$ is denoted  $I(B)$. 
An element $x\in A\setminus B$ is called \emph{externally active} if 
$x\in \clos \{ b \in B : b \le x \}$,
\ie $x$ is the maximal element of the unique circuit contained in $B\cup x$.
The set of externally active elements with respect to $B$ is denoted $E(B)$.\footnote{%
Usually, combinatorialists use $\min$ instead of $\max$ in the definition of the activities. 
 In the zonotopal algebra literature $\max$ is used. This has some notational advantages.
}
The \emph{Tutte polynomial}
\begin{align}
T_{(A,\BB)}(x,y) := \sum_{B\subseteq \BB} x^{\abs{I(B)}}y^{\abs{E(B)}} %
\end{align}
captures a lot of information about the matroid $(A,\BB)$.
One can show that it is independent of the order of the list $A$.

In this paper, we mainly  consider matroids that are realisable over some field $\K$.
Let $X=(x_1,\ldots, x_N)$ be a list of vectors spanning some $\K$-vector space $W$  and let $\BB(X)$ denote the set 
 of bases for $W$ that can be selected from $X$.
 One can easily see that $(X,\BB(X))$ is a matroid. The list $X$ is called a \emph{realisation} of this matroid 
 and a matroid $(A,\BB)$ is called \emph{realisable} 
 if there is a list of vectors $X$ and
 a bijection between $A$ and $X$ that induces a bijection between $\BB$ and $\BB(X)$.

A standard reference for matroid theory is Oxley's book \cite{MatroidTheory-Oxley}.
Survey papers on the Tutte polynomial are \cite{brylawski-oxley-1992,ellis-merino-2011}.

\smallskip
Throughout this paper, we use a running example, which we now introduce.
\begin{Example}
\label{Example:MatrixBases}
\begin{align}
\text{Let } X &:= \begin{pmatrix}
        1 & 0 & 1 \\ 0 & 1 & 1
      \end{pmatrix} = (x_1,x_2,x_3). %
\end{align}
The set of bases that can be selected from $X$ is
 $\BB(X) = \{ (x_1,x_2), (x_1,x_3), (x_2,x_3) \}$.
 The Tutte polynomial of the matroid $(X,\BB(X))$ is
 \begin{align}
  T_{(X,\BB(X))} (x,y) = x^2 + x + y.
 \end{align}
\end{Example}

\subsection{Some commutative algebra}
In this subsection, we define some commutative algebra terminology that is used in this paper. 

\begin{Definition}[A pairing between symmetric algebras]
We define the following pairing:
\begin{align}
 \pair{\cdot}{\cdot} : \R[s_1,\ldots, s_r] \times \R[t_1,\ldots, t_r] &\to \R   \\
    \pair{p}{f} := \left( p\left(\diff{t_1},\ldots, \frac{\partial}{\partial t_r} \right)  f \right) (0), 
\label{eq:pairingDefinition}
\end{align}
\ie we let $p$ act on $f$ as a differential operator and take the constant part of the result.
\end{Definition}
\begin{Remark}
One can easily show that the definition of the pairing $\pair{\cdot}{\cdot}$ %
 is independent of the choice of the bases for the symmetric algebras $\sym(U)$ and $\sym(V)$
 as long as the bases are dual to each other. 
\end{Remark}
\begin{Definition}
 Let $\Ical \subseteq \R[s_1,\ldots, s_r]$ be a homogeneous ideal. Its  \emph{kernel} or Macaulay inverse system 
 \cite{groebner-1938,abramson-2010,macaulay-1916}
  is defined as
 \begin{align}
   \ker \Ical &:= \{ f \in \R[t_1,\ldots, t_r] : \pair{q}{f} = 0 \text{ for all } q \in \Ical \}. 
 \end{align}
\end{Definition}

\begin{Remark}
 $\ker \Ical$ can also be written as
 \begin{align}
   \ker \Ical &:= \{ f\in \R[t_1,\ldots, t_r] : 
           p\left(\diff{t_1},\ldots, \frac{\partial}{\partial t_r} \right)  f  = 0  \} 
 \end{align}
 where $p$ runs over a set of generators for the ideal $\Ical$.
\end{Remark}

\begin{Remark}
For a homogeneous ideal $\Ical \subseteq \R[s_1,\ldots, s_r]$ of finite codimension 
the Hilbert series of $\ker \Ical$ and $\R[s_1,\ldots, s_r]/\Ical$ are equal.
For instance, this follows from  \cite[Theorem 5.4]{concini-procesi-book}.
\end{Remark}

A \emph{graded vector space} is a vector space $V$  that decomposes into a direct sum $V=\bigoplus_{i\ge 0} V_i$.
 A graded linear map $f: V\to W$ preserves the grade, \ie $f(V_i)$ is contained in $W_i$. 
For a graded vector space, we define its \emph{Hilbert series} as the formal power 
series $\hilb(V,t):=\sum_{i\ge 0} \dim(V_i) t^i$. 

Note that a linear map $f: V\to W$ induces an algebra homomorphism $ \sym(f) : \sym(V)\to \sym(W) $.

\subsection{Central zonotopal algebra}
\label{Subsection:CentralZA}

In this subsection, we define the Dahmen--Micchelli space
 $\Dcal(X)$ and its dual $\Pcal(X)$. 
The pair $(\Dcal(X), \Pcal(X))$, which is  called the central pair of zonotopal spaces in \cite{holtz-ron-2011},
 is the origin of zonotopal algebra.

A vector $u \in U$ naturally defines a polynomial $p_u \in \R[s_1,\ldots, s_r]$ as follows:
if $u$ can be expressed in the basis $(s_1,\ldots, s_r)$ as $u=\sum_{i=1}^r \lambda_i s_i$, then we define
$p_u:=\sum_{i=1}^r \lambda_i s_i  \in \R[s_1,\ldots, s_r]$.
For $Y\subseteq X$, we define $p_Y := \prod_{x\in Y} p_x$.
For the list $X$ in Example~\ref{Example:MatrixBases} we obtain $p_X=s_1s_2(s_1+s_2)$.
\begin{Definition}
\label{Definition:DahmenMicchelli}
\Xintro
 Then we define  
\begin{align}
  \Jcal(X) &:=  \ideal \{ p_T : T \subseteq X \text{ cocircuit}   \} \subseteq \R[s_1,\ldots, s_r] \\
  \text{and } 
  \Dcal(X) &:= \ker \Jcal(X) \subseteq \R[t_1,\ldots, t_r].
\end{align}
$\Dcal(X)$ is called the central $\Dcal$-space or \emph{Dahmen--Micchelli} space.
\end{Definition}

It can be shown that $\Dcal(X)$ is the space spanned by the local pieces of the box spline and their 
 partial derivatives. The
 box spline will be defined in the next section.
The space $\Dcal(X)$ was introduced in \cite{boor-hoellig-1982} and in 
\cite{dahmen-micchelli-1985} it was shown that its dimension is $\abs{\BB(X)}$.

\begin{Definition}
\Xintro
Then, we define  the central $\Pcal$-space
 \begin{align}
   \Pcal(X) &:= \spa \{ p_Y : Y\subseteq X, X\setminus Y \text{ has full rank} \} \subseteq \R[s_1,\ldots, s_r].  
  \end{align}
\end{Definition}

\begin{Proposition}[\cite{dyn-ron-1990}%
 ]
\label{Proposition:Pbasis}
\Xintro
  A basis for $\Pcal(X)$ is given by  
 \begin{align}
  \Bcal(X) &:= \{ Q_B   : B \in \BB(X) \},
 \end{align}
 where $Q_B:=p_{ X\setminus (B \cup E(B))}$.
\end{Proposition}
The space $\Pcal(X)$ can also be written as the kernel of an ideal.
The following proposition appeared in \cite{ardila-postnikov-2009}
 and earlier in a slightly different version in \cite{boor-dyn-ron-1991}.
\begin{Proposition}
\Xintro
Then,
\begin{align}
 \Pcal(X)&= \ker \Ical(X), \\
 \text{where }  \Ical(X) &:= \ideal\left\{ p_\eta^{m(\eta)} : \eta\in V\setminus  \{ 0 \} \right\}  \subseteq \R[t_1,\ldots, t_r]
 \end{align}
 and $m : V \to \N$ assigns to $\eta\in V$ the number of vectors in $X$ that are \emph{not} perpendicular to $\eta$.
\end{Proposition}

\begin{Example}
 Let $X$ be the list of vectors we defined in Example~\ref{Example:MatrixBases}. %
 Then
\begin{align}
  \Dcal(X) &= \ker ( \ideal \{ s_1s_2, s_1(s_1+s_2), s_2(s_1+s_2)\} ) = \spa\{ 1, t_1, t_2 \}, \\
  \Ical(X) &= \ideal\{ t_1^2, t_2^2, (t_1 - t_2)^2 \} + \R[t_1,t_2]_{\ge 3} = \ideal\{ t_1^2, t_2^2, t_1t_2 \}, \\
  \text{and }\Pcal(X) &= \ker\Ical(X) = \spa \{ 1, s_1, s_2 \}.
\end{align}

\end{Example}

\begin{Proposition}[\cite{dyn-ron-1990,jia-1990}] %
\label{Proposition:PDduality}
\Xintro
Then the spaces $\Pcal(X)$ and $\Dcal(X)$ are dual under the pairing $\pair{\cdot}{\cdot}$, \ie
\begin{align}
 \Dcal(X) &\to \Pcal(X)^* \\
  f &\mapsto \pair{\cdot}{f}
\end{align}
is an isomorphism.
\end{Proposition}
The preceding proposition implies that the Hilbert series of $\Pcal(X)$ and $\Dcal(X)$ are equal.
By Proposition~\ref{Proposition:Pbasis}, this Hilbert series
is a matroid invariant and a specialisation of the Tutte polynomial.
These facts are summarised in the following proposition.
\begin{Proposition} %
\label{Proposition:CentralTutteEval}
Let $X\subseteq U\cong \R^r$ be a list of vectors $N$ vectors that spans $U$. Then
 \begin{align}
  \hilb(\Dcal(X),q) = \hilb(\Pcal(X),q) &=  q^{N-r} T_{(X,\BB(X))}(1,\frac 1q) = \sum_{B\in \BB(X)} q^{N-r - \abs{E(B)}}.
 \end{align}
\end{Proposition}

\begin{Remark}
 Most of the results mentioned above also  
 hold over other fields of characteristic zero or even over arbitrary fields.
 However, in this paper we work over the reals, because the construction in Section~\ref{Section:BasisElementsConstruction}
 uses distributions.
%
%
%
\end{Remark}

\subsection{Organisation of the article}
The remainder of this article is organised as follows.
In Section~\ref{Section:Background} we provide some additional mathematical background, in particular on splines.

In Section~\ref{Section:BasisElementsConstruction}
 we construct certain polynomials $R^B$ as convolutions of differences of multivariate splines.
In Section~\ref{Section:MainCentral} we show that
\begin{align}
 \Bcyr(X) :=  \{ \abs{\det(B)} R^B : B\in \BB(X)   \} 
\end{align}
is a basis for $\Dcal(X)$ and we prove that this basis is dual 
to the basis $\Bcal(X)$ for $\Pcal(X)$.
In Section~\ref{Section:DelConExactSequencesCentral} we discuss deletion-contraction and two short exact sequences.
In Section~\ref{Section:Combinatorics} %
 we introduce a new combinatorial structure called forward exchange matroid. 
 This is an ordered matroid together with a subset $\BB'$ of its set of bases 
 with the so-called forward exchange property.
 In Section~\ref{Section:GeneralZonotopalSpaces} we 
 introduce the generalised $\Pcal$-space  
 $\Pcal(X,\BB'):= \spa \{ Q_B   : B \in \BB' \}$ and the generalised $\Dcal$-space $\Dcal(X,\BB')$.
 We show that 
most of the results that we described in Subsection~\ref{Subsection:CentralZA}
 and Section~\ref{Section:MainCentral}
 still hold for these  spaces if $\BB'$ has the forward exchange property.
For example, the two spaces are dual and
 a suitable subset of $\Bcyr(X)$ turns out to be a basis for $\Dcal(X,\BB')$.
Furthermore, $\Dcal(X,\BB')$ and $\Pcal(X,\BB')$ have deletion-contraction decompositions
 that are related to the deletion-contraction reduction of the Tutte polynomial.

In Section~\ref{Section:ZonotopalAlgebra} we review the previously known zonotopal spaces 
 and we show that they are  special cases of our spaces $\Dcal(X,\BB')$ and $\Pcal(X,\BB')$.

\subsection*{Acknowledgements}
 The author would like to thank Olga Holtz 
 and  Martin G\"otze for discussions and   helpful suggestions. He is grateful to 
 Zhiqiang Xu for pointing out the paper \cite{xu-2011}.

\section{Preliminaries}
\label{Section:Background}
In this section we provide some mathematical background.
In Subsection~\ref{Subsection:DiscreteGeometry} we define some objects from discrete geometry.
In Subsection~\ref{Subsection:Distributions} we review distributions and
in Subsection~\ref{Subsection:Splines} we discuss box splines and multivariate splines.
Subsection~\ref{Subsection:LeastMap} contains information about previously known algorithms for the
 construction of bases
 for $\Dcal$-spaces.

\subsection{Cones and zonotopes}
\label{Subsection:DiscreteGeometry}
\begin{Definition}
 Let $X=(x_1,\ldots, x_N)\subseteq U \cong \R^r$ be a list of vectors. Then we define 
the \emph{zonotope} $Z(X)$ and the \emph{cone} $\cone(X)$ by
 \begin{align}
 Z(X):= \left\{ \sum_{i=1}^N \lambda_i x_i : 0\le \lambda_i \le 1  \right\}
 \text{ and }
 \cone(X) := \left\{ \sum_{i=1}^N \lambda_i x_i : \lambda_i \ge 0 \right\}.
 \end{align}
\end{Definition}
Zonotopes are closely connected to zonotopal algebra. For example,  if $X$ is 
unimodular, then
 $\vol(Z(X))=\dim\Pcal(X)=\dim\Dcal(X)$ and the space $\Pcal(X)$ can be obtained from the
 set of lattice points in the half-open zonotope via %
least map interpolation that is explained 
 in Subsection~\ref{Subsection:LeastMap}.

\subsection{Distributions}
\label{Subsection:Distributions}
In this paper we are mainly interested in multivariate polynomials.
However, in the construction in Section~\ref{Section:BasisElementsConstruction} more general
 objects appear in  intermediate steps.
 In this construction, we need ''generalised polynomials`` whose support is contained in a subspace.
 Furthermore, we use convolutions and %
 the fact that convolutions and partial derivatives commute.
 Distributions have all of the desired properties.
 In this subsection we summarise  important facts about distributions
 that we will need later on. For a detailed introduction to the subject, we refer the reader to 
 Laurent Schwartz's book \cite{schwartz-1966}.

A \emph{distribution} on a vector space $U\cong \R^r$ (or an open subset of $U$) 
 is a continuous linear functional that maps a test function to a real number. 
 \emph{Test functions} are compactly supported smooth functions $U\to \R$.
An important example of a distribution is the delta distribution $\delta_x$ given by $\delta_x(\varphi):=\varphi(x)$.
 A locally integrable function $f : U \to \R$  defines a distribution $T_f$ in the following way:
\begin{align}
  T_f (\varphi) := \int_{U} f(u) \varphi(u) \, \di u.  
\end{align}
Recall that for two functions $f,g: U \to \R$, the \emph{convolution} is defined as
\begin{align}
 f*g:=\int_U  f(u)g(\cdot - u) \,\di u.
\end{align}
This is well-defined only if $f$ and $g$ decay sufficiently rapidly at infinity in order for the integral to exist.
The convolution of two distributions can also be defined under certain conditions.

A distribution $T$ vanishes on a set $\Gamma\subseteq U$ if $T(\varphi)=0$ for all test functions whose support is contained in 
 $\Gamma$.
The \emph{support} $\supp(T)$ of  $T$ is the complement of the maximal open set on which $T$ vanishes.

Let $S_\xi$ and $T_\eta$ be two distributions 
for which  %
  $\supp(S_\xi) \cap (K - \supp(T_\eta))$ is compact for any compact set $K$.
Let $\varphi: U \to \R$ be a test function with support $K$. Then, we define the convolution 
\begin{align}
  (S_\xi * T_\eta)(\varphi) %
   := S_\xi ( T_\eta ( \alpha(\xi) \varphi(\xi+\eta))),
\end{align}
where $\alpha$ denotes a test function that  equals $1$ on a neighbourhood of 
 $ \supp(S_\xi) \cap (K - \supp(T_\eta)) $. 
When evaluating $ T_\eta (\alpha(\xi) \varphi( \xi + \eta ) ) $, we think of $\varphi(\xi+\eta)$ as a function in $\eta$ 
 and of $\xi$ as a fixed parameter. %
 Then, $ T_\eta ( \alpha(\xi) \varphi( \xi + \eta ) )$ is a function in $\xi$ with compact support
 that is contained in $K-\supp(T_\eta)$.
 Note that the definition of $(S_\xi * T_\eta)(\varphi)$ is independent of the choice of the function $\alpha$.
 The multiplication by $\alpha$ is necessary to ensure that
   $T_\eta ( \alpha(\xi) \varphi( \xi + \eta ) )$ as a function in $\xi$ has compact support.

Note that the convolution of two distributions is a commutative operation and $T*\delta_0 = T$.
Let $u\in U$. 
 The partial derivative of a distribution $T$ in direction $u$ is defined by
 $(D_u T)(\varphi) := -T (D_u\varphi)$.
Convolutions of distributions have the same nice property with respect to partial derivatives as
 convolutions of functions. Namely,
 if $T_1$ and $T_2$ are distributions on $U$ and $u\in U$, then
\begin{align}
\label{eq:ConvDiffCommute}
  D_u(T_1*T_2) = (D_u T_1) * T_2 = T_1 * (D_u T_2).
\end{align}

\subsection{Splines}
\label{Subsection:Splines}

In this subsection we introduce multivariate splines and box splines as in
 \cite[Chapter 7]{concini-procesi-book}. Another good reference is \cite{BoxSplineBook}.
\begin{Definition}
\label{Definition:Splines}
 Let $X\subseteq U\cong \R^r$ be a finite list of vectors.
The \emph{multivariate spline} (or truncated power) $T_X$ and the \emph{box spline} $B_X$ are distributions that are characterised 
 by the formulae
\begin{align}
 \int_{U}  f(u) B_X(u) \,\di u &= \int_0^1 \cdots \int_0^1 f 
             \left( \sum_{i=1}^N \lambda_ix_i \right) \di \lambda_1 \cdots \di \lambda_N  \\
 \text{and } \int_{U}  f(u) T_X(u) \,\di u &= \int_0^\infty \cdots \int_0^\infty 
	      f \left( \sum_{i=1}^N \lambda_ix_i \right) \di \lambda_1 \cdots \di \lambda_N.
\end{align}
\end{Definition}
The multivariate spline is well-defined only if the convex hull 
 of the vectors in $X$ does not contain $0$
   or equivalently, if there is a functional 
 $\varphi\in V$ \st $\varphi(x)>0$ for all $x\in X$.
 If all vectors are non-zero,
 it is of course always possible to multiply  
 certain entries of the list $X$ by $-1$ \st this condition is satisfied.
 Note that in Definition~\ref{Definition:Splines}, we do 
 \emph{not} require that $X$ spans $U$ in contrast to most of the rest of this paper.

$B_X$ and $T_X$ can be identified with the functions
\begin{align}
B_X(u) &= \frac{1}{\sqrt{\det(XX^T)}}\vol\nolimits_{N- \dim(\spa(X))} \left(\{ z \in [0;1]^N : Xz = u \}\right)  
\label{eq:BoxSplineVolumeFormula}
\\
\text{and } T_X(u) &= \frac{1}{\sqrt{\det(XX^T)}}\vol\nolimits_{N- \dim(\spa(X))} \left( \{ z \in \R^N_{\ge 0} : Xz = u \}\right). 
\label{eq:MultSplineVolumeFormula}
\end{align}
It follows immediately from \eqref{eq:BoxSplineVolumeFormula} and \eqref{eq:MultSplineVolumeFormula}
 that $B_X$ is supported in the zonotope $Z(X)$ and $T_X$ is supported in the cone $\cone(X)$.
For a basis $C\subseteq U$,
\begin{align}
  B_C  = \frac{\chi_{Z(C)}}{\abs{\det(C)}}  \text{ and }
  T_C  = \frac{\chi_{\cone(C)}}{\abs{\det(C)}}.
\label{eq:multivariateBasis}
\end{align}

\begin{Remark}
 The box spline can easily be obtained from the  multivariate spline. 
 Namely,
\begin{align}
B_X(x) =\sum_{S\subseteq X} (-1)^{\abs S} T_X\left(x- a_S\right),
\end{align}
where $a_S:=\sum_{a\in S} a$.

From now on, we will only consider $T_X$. We introduced the box spline only because of its importance in approximation theory.
\end{Remark}

\begin{Theorem}%
\label{Theorem:SplineLocallyPolynomial}
 Let $X\subseteq U\cong \R^r$ be a finite list of vectors that spans $U$ and whose convex hull does not contain $0$.
The cone $\cone(X)$ can be decomposed into finitely many cones $C_i$ \st
 $T_X$ restricted to each $C_i$ is a homogeneous polynomial of degree $N-r$.

\end{Theorem}

\begin{Theorem}
 The space spanned by the local pieces of the multivariate spline $T_X$ and their
  partial derivatives is equal to the  Dahmen--Micchelli space $\Dcal(X)$ that was defined 
 in Definition~\ref{Definition:DahmenMicchelli}.
\end{Theorem}
The multivariate spline can also be defined inductively by
 the  convolution formula
\begin{align}
 T_{(X,x)} &=  T_X * T_{(x)} = \int_0^\infty T_X(\cdot - \lambda x) \,\di \lambda 
\label{eq:multivariateInductive}
\end{align}
using \eqref{eq:multivariateBasis} as a starting point.
In particular, 
  $T_X = T_{x_1} * \cdots * T_{x_N}$. 
Since $D_x T_{x} = \delta_0$, the convolution formula implies
 for $Y \subseteq X$ that 
\begin{align}
\label{eq:SplineDiff}
 D_Y T_X = T_{X\setminus Y}, \text{ where } D_Y:=\prod_{x\in Y} D_x.
\end{align}

\begin{Example}
\label{Example:FirstSplineCalculation}
We consider the same list $X$ as in Example~\ref{Example:MatrixBases}.
By \eqref{eq:multivariateBasis}, $T_{(x_1,x_2)}$ is the indicator function of $\R^2_{\ge 0}$.
Then, by \eqref{eq:multivariateInductive}, we can deduce
\begin{align}
 T_X(s_1,s_2) = \int_0^\infty \chi_{\R^2_{\ge 0}}( s_1- \lambda, s_2- \lambda) \, \di \lambda = \min(s_1,s_2)
\end{align}
for $s_1,s_2\ge 0$.
See Figure~\ref{Figure:Splines} for a graphic description of $T_X$.
\end{Example}
\begin{figure}
\begin{center}
\input{BoxSpline3.pspdftex}
\hspace*{4.5mm}
\input{MultivariateSpline.pspdftex}
\end{center}
\caption{The box spline and the multivariate spline defined by the list $X$ in Example~\ref{Example:MatrixBases}.
The multivariate spline is calculated in Example \ref{Example:FirstSplineCalculation}. On the right, the dashed lines
 are level curves.
\\
 The image on the left hand side is taken from Matthias Lenz, Interpolation, Box Splines, and Lattice Points in Zonotopes, 
International Mathematics Research Notices, 2014, vol.~2014, no.~20, p.~5699, \url{http://dx.doi.org/10.1093/imrn/rnt142},
by permission of Oxford University Press. The image on the right hand side was created by the author in 2012.
It is made available under the Creative Commons Attribution 4.0 International License.
To view a copy of this license, visit \url{http://creativecommons.org/licenses/by/4.0/}.
}
\label{Figure:Splines}
\end{figure}
\begin{Example}
Let  $X_i:=(\underbrace{1,\ldots, 1}_{i \text{ times}})$. Then,
\begin{align}
 T_{X_1}(s) &= \chi_{\R_{\ge 0}}(s) \\
\text{and}\quad T_{X_{i+1}}(s) &= \int_0^\infty T_{X_{i}}(s- \lambda) \,\di \lambda 
     = \int_0^s \frac{\lambda^{i-1}}{(i-1)!} \,\di \lambda= \frac{s^{i}}{i!} \text{ for } s\ge 0.
\end{align}
 \end{Example}

\subsection{Previously known methods for constructing bases for $\Dcal$-spaces}
\label{Subsection:LeastMap}
Two other methods are known to construct a basis for $\Dcal(X)$. 
 However, our algorithm has several advantages over the other two: 
 it is canonical, \ie
 it only depends on the order of the list $X$ 
 and it yields a basis that is dual to
 the known basis $\Bcal(X)$ for the $\Pcal$-space.

In Wolfgang Dahmen's construction \cite{dahmen-1990}, polynomials are chosen as basis elements that are local pieces of 
 certain multivariate splines. 
 For certain choices of the parameters in his construction, it might yield the same basis as ours.

We describe the second construction in more detail.
It uses the so-called least map interpolation 
that was introduced by  Carl de Boor and Amos Ron \cite{boor-ron-1991}. 
Given a finite set $S\subseteq V$, they construct 
  a space of polynomials   $\Pi(S)\subseteq \sym(V)$ of dimension $\abs{S}$
 with certain nice properties.
 Recall that $U\cong \R^r$, $V$ denotes the dual space and
 a vector $v\in V$ defines a linear form $p_v\in \R[t_1,\ldots, t_r]\cong \sym(V)$. 
 We define the exponential function as usual by
 \begin{align}
   e^{v} := \sum_{j\ge 0} \frac{ p_{v}^j}{j!} \in \R[[t_1,\ldots, t_r]] \cong \sym(U)^*.    
 \end{align}
 The least map $\downarrow$ maps a non-zero element of the ring of formal power series $\R[[t_1,\ldots, t_r]]$ to 
 its homogeneous component of  lowest degree that is non-zero. Furthermore, $0_\downarrow:=0$.
  The \emph{least space} of a finite set $S\subseteq V$ is defined as
\begin{align}
 \Pi(S) := \spa \{  f_\downarrow   : f\in\spa \{ e^{v} : v\in S \} \} \subseteq \R[t_1,\ldots, t_r].
\end{align}
Let $x\in U$ and $c_x\in \R$. 
 This defines a \emph{hyperplane} 
\begin{align}
 H_{x,c_x}:= \{ v\in V : v(x) = c_x  \}.
\end{align}
If we fix a vector $c\in\R^X$, we obtain a 
 hyperplane arrangement $\Hcal(X,c) = \{ H_{x,c_x} : x\in X\}$.

Every basis $B\subseteq X$ determines a unique vertex
$\theta_B\in V$ of the hyperplane arrangement $\Hcal(X,c)$ that satisfies
 $\theta_B(x)=c_x $ for all $x\in B$. 
 In matrix notation,
$\theta_B= B^{-1} c_B$, where $c_B$ denotes the restriction of $c$ to $\R^B$.
If the vector $c$ is sufficiently generic,
  then $\theta_B \neq \theta_{B'}$ for distinct bases $B$ and $B'$.
  In this case, the hyperplane arrangement $\Hcal(X,c)$ is said to be in general position. 
 For more information on hyperplane arrangements, see \cite{stanley-2007}.

The following surprising theorem makes a connection between hyperplane arrangements and the space $\Dcal(X)$.
 It generalises to other $\Dcal$-spaces (see \cite{holtz-ron-2011,holtz-ron-xu-2012,li-ron-2013}).
\begin{Theorem}[\cite{boor-ron-1991}]
\label{Theorem:HyperplaneArrangementD}
\Xintro
Let $c\in\R^X$ be a vector \st the hyperplane arrangement  $\Hcal(X,c)$
 is in general position and let  $S$ be the set of vertices of $\Hcal(X,c)$.
 Then
\begin{align}
\Dcal(X)=\Pi(S).
\end{align}
\end{Theorem}

De Boor and Ron gave a method to select 
a basis from
$\Pi(S)$ in  \cite{deBoor-Ron-computational-1992} 
(see   \cite[Chapter II]{BoxSplineBook} for a summary).
Their construction depends on the choice of the vector $c$, an ordering of the bases and an ordering
 of $\N^r$ while our construction only depends on the order on $X$. %

\begin{Example}
 This is a continuation of Example~\ref{Example:MatrixBases}.
\\
\begin{minipage}{0.74\textwidth}
Let  $c_1=c_2=0$ and $c_3=1$. The set of vertices of
 $\Hcal(X,c)$ is $S=\{(0,0), (1,0), (0,1)\}$.
Then, 
\begin{align*}
 \Pi(S) &=  \spa \{ f_\downarrow : f\in \spa \{ 1, e^{t_1}, e^{t_2} \}\}  \\
        &=  \spa \{1, t_1, t_2\}, \text{ since } 1=1_\downarrow,\; \\
 &\qquad  t_1= (e^{t_1}-1)_\downarrow,\text{ and } t_2= (e^{t_2}-1)_\downarrow.
\end{align*}
\end{minipage}
\begin{minipage}{0.25\textwidth}
\hfill
 \begin{picture}(0,0)%
\includegraphics{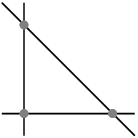}%
\end{picture}%
\setlength{\unitlength}{3108sp}%
\begingroup\makeatletter\ifx\SetFigFont\undefined%
\gdef\SetFigFont#1#2#3#4#5{%
  \reset@font\fontsize{#1}{#2pt}%
  \fontfamily{#3}\fontseries{#4}\fontshape{#5}%
  \selectfont}%
\fi\endgroup%
\begin{picture}(1394,1394)(33279,-13808)
\put(33566,-13493){\makebox(0,0)[lb]{\smash{{\SetFigFont{8}{9.6}{\rmdefault}{\mddefault}{\updefault}{\color[rgb]{0,0,0}$(0,0)$}%
}}}}
\put(33618,-12661){\makebox(0,0)[lb]{\smash{{\SetFigFont{8}{9.6}{\rmdefault}{\mddefault}{\updefault}{\color[rgb]{0,0,0}$(0,1)$}%
}}}}
\put(34405,-13466){\makebox(0,0)[lb]{\smash{{\SetFigFont{8}{9.6}{\rmdefault}{\mddefault}{\updefault}{\color[rgb]{0,0,0}$(1,0)$}%
}}}}
\end{picture}%

\end{minipage}
\end{Example}

\section{Construction of basis elements}
\label{Section:BasisElementsConstruction}

In this section we construct a polynomial $R_Z^B$ in $\R[s_1,\ldots, s_r]$, given
  a finite list 
 $Z\subseteq U \cong \R^r$ and a basis $B\subseteq Z$.
  Later on we show that polynomials of this type form bases for various zonotopal $\Dcal$-spaces 
 if one chooses suitable pairs $(B,Z)$.
 The polynomial $R^B_Z$ is constructed as a convolution of differences of multivariate splines. 

\smallskip
 Let $Z\subseteq U$ be a finite list and let $B=(b_1,\ldots, b_r) \subseteq Z$ be a basis. 
 It is important that the basis is ordered and that this order is
  the order obtained by restricting
 the order on $Z$ to $B$. For $i\in\{ 0,\ldots, r\}$, we define $S_i = S_i^B := \spa \{ b_1,\ldots, b_i \}$. 
Hence,
 \begin{align}
 \label{eq:flagofsubspaces}
   \{ 0 \} = S_0^B  \subsetneq S_1^B \subsetneq S_2^B \subsetneq \ldots \subsetneq S_r^B = U \cong \R^r
 \end{align}
is a flag of subspaces.
 We define an orientation on each of the spaces $S_i$ by saying that $( b_1, \ldots, b_i )$ is a positive basis 
 for $S_i$. 
 Now a basis $D=(d_1,\ldots, d_i)$ for $S_i$ is called positive if the map that 
 sends $b_\nu$ to $d_\nu$ for $1\le \nu\le i$ has positive determinant.  
 Let $u\in S_{i}\setminus S_{i-1}$. If
 $(b_{1},\ldots, b_{i-1},u)$ is a positive basis, we call $u$ \emph{positive}. Otherwise, we call $u$ \emph{negative}.  
 Equivalently, $u\in S_{i}\setminus S_{i-1}$ is positive if and only if the $b_i$ coefficient is positive when we express $u$
 as a linear combination of $B$.

We partition $Z \cap (S_{i} \setminus S_{i-1})$ as follows:
 \begin{align}
   P_{i}^B &:= \{  u \in Z \cap (S_{i} \setminus S_{i-1}) : u \text{ positive}\} \\  
 \text{and }   N_{i}^B &:= \{ u \in Z \cap (S_{i} \setminus S_{i-1}) : u \text{ negative} \}.
\intertext{We define}
  T_{i}^{B+}  &:= (-1)^{\abs{N_{i}}} \cdot T_{P_i} * T_{-N_i} 
 \text{ and } T_{i}^{B-}:= (-1)^{\abs{P_i}} \cdot T_{-P_i} * T_{N_i}.
 \end{align}
Note that  $T_{i}^{B+}$ is supported in  $\cone(P_i,-N_i)$ and 
 that 
 \begin{align}
 T_{i}^{B-}(x) = (-1)^{\abs{P_i\cup N_i}} T_i^{B+}(- x). 
 \end{align}
Now define 
\begin{align}
   R_{i}^{B}  := T_i^{B+} -  T_i^{B-} %
 \qquad \text{ and } \qquad
 R^B_Z = R^B := R_1^B * \cdots * R_r^B.
\label{eq:BasisElementConvolution}
\end{align}
For an example of this construction see Example~\ref{Example:PandDbases} and 
 Figure~\ref{Figure:ConvolutionCones}.
In Corollary~\ref{Corollary:HomPoly},
 we will see that the distribution $R^B_Z$ can be identified with a homogeneous polynomial.

\begin{figure}[tbp]
\begin{center}
\input{ConvolutionCones.pspdftex}
\end{center}
\caption{The geometry of the construction of the polynomial $R^{(x_1,x_3)}_X$ 
 in Example~\ref{Example:PandDbases}.
 Note that $N_2=\emptyset$.
\\
 This image was created by the author in 2012.
It is made available under the Creative Commons Attribution 4.0 International License.
To view a copy of this license, visit \url{http://creativecommons.org/licenses/by/4.0/}.
 }
\label{Figure:ConvolutionCones}
\end{figure}

\begin{Remark}
A similar construction of certain quasi-polynomials in the discrete case is done in \cite[Section 3]{deconcini-procesi-vergne-2010b}
 (see also \cite[Section 13.6]{concini-procesi-book}).
 The part of Theorem~\ref{Theorem:MainTheoremCentral} 
 that exhibits a basis for $\Dcal(X)$  
 can be seen as a special case of  Theorem~3.22 in \cite{deconcini-procesi-vergne-2010b}.
\end{Remark}

\begin{Remark}
 The construction of the polynomials $R_Z^B$ may at first seem rather complicated in comparison
 with construction of the polynomials $Q_B$ that form bases of the $\Pcal$-spaces.

 Here are a few remarks to explain this construction: 
multivariate splines are very convenient because it is so easy to calculate their partial derivatives 
(cf.~\eqref{eq:SplineDiff}).
 Taking differences of two splines in the definition of $R_i^B$ ensures that $R^B_Z$ is a polynomial 
 and not just piecewise 
 polynomial.
 In fact, $R_1^B * \ldots * R_i^B$ is a ``polynomial supported in $S_i$'' for all $i$.

 We have to change the sign of some of the vectors
 before constructing the multivariate spline  $T_{i}^{B+}$ to  ensure that 
 all the convolutions are well-defined.
 For example,  the convolutions
 in \eqref{eq:BasisElementConvolution} are well-defined for the following reason:
 the support of $R_1^B*\cdots *R_i^B$ is contained in $S_i$.
 The support of $R_{i+1}^B$ is $\cone(P_{i+1},-N_{i+1}) \cup \cone(-P_{i+1},N_{i+1})$.
For every compact set $K$, the set
\begin{align}
S_i \cap ( K- (\cone(P_{i+1},-N_{i+1}) \cup \cone(P_{i+1},-N_{i+1})) )
\end{align}
 is compact.
\end{Remark}

\begin{Proposition}
\label{Proposition:BasisElementIsLocalPiece}
 The distribution $R^B_Z$ is a local piece of the multivariate spline $T_1^{B+} * \cdots * T_r^{B+}$. %
\end{Proposition}
\begin{proof}
Let $c\gg 0$ and let
\begin{align}
\label{equation:VectorInTheBigCell}
  \tau :=  b_1 + \frac 1c b_2 + \ldots + \frac 1{c^{r-2}} b_{r-1} + \frac 1{c^{r-1}} b_r. 
\end{align} 
See Figure~\ref{Figure:ConeConstruction} for an example of this construction.
The vector $\tau$ is contained in $\cone(Z)$. 
By Theorem~\ref{Theorem:SplineLocallyPolynomial}, 
there exists a subcone of $\cone(Z)$ that contains $\tau$ \st
$T_Z$ agrees with a polynomial $p_{\tau,Z}$ on this subcone.
 We claim that $R^B_Z$ is equal to $p_{\tau,Z}$. 
Note that
\begin{align*}
 R^B_Z  &= (T_1^{B+} - T_1^{B-}) * \cdots * (T_r^{B+} - T_r^{B-}) 
    = \sum_{J \subseteq [r]} (-1)^{\abs J + \sum_{i \not\in J}\abs{N_i} + \sum_{i \in J}\abs{P_i} } T_{Z^J_B},
\intertext{where}
 Z^J_B &=  %
   \bigcup_{i\not \in J} (P_i,-N_i) %
   \cup \bigcup_{i \in J} (-P_i,N_i). 
\end{align*}
In order to prove our claim, it is sufficient to show that
 $\tau$ is contained in $\cone(Z^J_B)$ if and only if $J=\emptyset$.
The ``if'' part is clear.

\begin{figure}[tbp]
\begin{center}
\input{ConeConstruction3.pspdftex}
\end{center}
\caption{The setup in the proof of Proposition~\ref{Proposition:BasisElementIsLocalPiece}. Here, $J=\emptyset$.
\\
 This image was created by the author in 2012.
It is made available under the Creative Commons Attribution 4.0 International License.
To view a copy of this license, visit \url{http://creativecommons.org/licenses/by/4.0/}.
}
\label{Figure:ConeConstruction}
\end{figure}

Let $J$ be non-empty and let  $j^*$ be the minimal element. 
For $\alpha\in \R$ let $\phi_\alpha : U \to \R$ be the linear form that maps a vector $x$ to
\begin{align}
  \sum_{j=j^*}^r (-1)^{\chi_J(j)} \alpha^{j} \lambda_j(x),
\end{align}
where $\lambda_j(x)$  denotes the coefficient of $b_j$ when $x$ is written in the basis
 $(b_1,\ldots, b_r)$.
We claim that for sufficiently large $\alpha$,  
$\phi_\alpha$ is non-negative on
$Z^J_B$ and $\phi_\alpha(\tau)<0$. By Farkas' Lemma (\eg \cite[Section 5.5]{schrijver-co-volA}), this proves 
 that $\tau$ is not contained in  $\cone(Z^J_B)$.

If $x\in S_i \cap Z_B^J$ for $i< j^*$, then obviously $\phi_\alpha(x)=0$.  
If $x\in (S_i\setminus S_{i-1}) \cap Z_B^J $ for $i\ge j^*$, then 
 $\lambda_i(x)\neq 0$ and $\lambda_\nu(x)=0$ for all $\nu\ge i+1$.
In addition, $(-1)^{\chi_J(i)}\lambda_i(x)>0$, since all vectors in $(P_i,-N_i)$ have a positive $b_i$ component when written
 in the basis $(b_1,\ldots, b_r)$ and all vectors in $(-P_i,N_i)$ have a negative $b_i$ component.
Hence,
 $\phi_\alpha(x) = (-1)^{\chi_J(i)} \alpha^{i} \lambda_i(x) 
  + o(\alpha^{ i } ) = \alpha^{i} \abs{\lambda_i(x)} + o(\alpha^{i} )$.
 This is positive for sufficiently large $\alpha$. 
 
 Since $(-1)^{\chi_J(j^*)}=-1$  we obtain
\begin{align}
\phi_{\alpha}(\tau)= -\frac{ \alpha^{j^*} } {c^{j^*-1} } \pm \frac{\alpha^{j^*+1}}{c^{j^*}} \pm \ldots 
  = -\frac{\alpha^{j^*}}{c^{j^*-1}} + o\left(\frac{1}{c^{j^*-1}}\right).
\end{align}
This is negative for sufficiently large $c$. Note that we fix a large $\alpha$ first and then we let $c$ grow.
\end{proof}
 Note that the distribution $R^B_Z$ does not change if we add or remove zero vectors from the list $Z$.
 Using Theorem~\ref{Theorem:SplineLocallyPolynomial}, we can deduce 
 from Proposition~\ref{Proposition:BasisElementIsLocalPiece}
 the following corollary.
\begin{Corollary}
\label{Corollary:HomPoly}
 Let $\tilde Z$ be the list of vectors obtained from $Z$ by removing all copies of the zero vector.
 The distribution $R^B_Z=R^B_{\tilde Z}$ can be identified with a homogeneous polynomial of degree $|\tilde Z|-r$.
\end{Corollary}

\begin{Remark}
The local pieces of the multivariate spline are uniquely determined by a certain equation
 (cf.~{\cite[Theorems 9.5 and 9.7]{concini-procesi-book}}). Taking into account
 Proposition~\ref{Proposition:BasisElementIsLocalPiece}, this 
 gives us a different method to 
 calculate the polynomials $R^B_Z$.
\end{Remark}

The following theorem that is due to Zhiqiang Xu 
 yields another formula for the polynomials $R_Z^B$.
 It is a variant of Brion's formula
\cite{brion-1988}.

\begin{Theorem}[{\cite[Theorem 3.1.]{xu-2011}}]
\label{Theorem:XuSplineFormula}
\XintroCard
Let $c\in\R^X$ be a vector \st the hyperplane arrangement  $\Hcal(X,c)$
 is in general position.
     For a basis $B\in \BB(X)$, let $\theta_B \in V$ denote the vertex of $\Hcal(X,c)$  corresponding to $B$
(cf.~Subsection~\ref{Subsection:LeastMap}). 
Then
 \begin{align}
\label{eq:XuSplineFormula}
   T_X(u) &= \frac{1}{(N-r)!} \sum_{B\in\BB(X)} 
                \frac{ (-\theta_B u)^{N-r} }{ \abs{\det(B)} \prod_{x \in X \setminus B  }
            ( \theta_B x - c_x)} \chi_{\cone(B)}(u).
 \end{align}
\end{Theorem}
 Note that the numerator \eqref{eq:XuSplineFormula} is non-zero because 
 $c$ is chosen \st  $\Hcal(X,c)$ is in general position. 
Using Proposition~\ref{Proposition:BasisElementIsLocalPiece}, one can deduce the following corollary.

\begin{Corollary}
 Let $Z\subseteq U\cong\R^r$ be a list of $n$ vectors that spans $U$.
 Let $c$ and $\theta_B$ as in Theorem~\ref{Theorem:XuSplineFormula}.  
Then,
 the polynomial $R_Z^B(u)$ is given by
  \begin{align}
   R_Z^B (u)  &= \frac{1}{(n-r)!} \sum_{ \substack{B'\in\BB(Z^{B+})\\\tau \in \cone(B') } } 
                \frac{ (-\theta_{B'} u)^{n-r} }{ \abs{\det(B')} \prod_{x \in Z \setminus B'  }( \theta_{B'} x - c_x)},
 \end{align}
 where $\tau$ denotes the vector defined in \eqref{equation:VectorInTheBigCell} and $Z^{B+}$
 denotes the reorientation of the list $Z$ \st all vectors are positive with respect to $B$, \ie $Z^{B+}= \bigcup_{i=1}^r (P_i, -N_i)$.
\end{Corollary}

\section{A basis for the Dahmen--Micchelli space $\Dcal(X)$}
\label{Section:MainCentral}

In this section we define a set $\Bcyr(X)$ 
 and we show that this set is a basis for the central $\Dcal$-space $\Dcal(X)$. Furthermore, we show that this basis
 is dual to the basis $\Bcal(X)$ of the central $\Pcal$-space $\Pcal(X)$.
 Note that $\Bcyr$ is the equivalent of the letter $B$ in the Cyrillic alphabet. 

\begin{Definition}[Basis for $\Dcal(X)$]
\Xintro
Recall that $\BB(X)$ denotes the set of bases that can be selected from $X$
 and that $E(B)$ denotes the set of externally active elements with respect to a basis $B$.
We define
\begin{align}
 \Bcyr(X) :=  \{ \abs{\det(B)}  R^B_{X\setminus E(B)} : B\in \BB(X)   \}.
\end{align}
\end{Definition}

\begin{Theorem}
\label{Theorem:MainTheoremCentral}
\Xintro
 Then
$\Bcyr(X)$ %
is a basis for the central Dahmen--Micchelli space $\Dcal(X)$ and this 
 basis is dual to the basis $\Bcal(X)$ for the central $\Pcal$-space  $\Pcal(X)$. %
\end{Theorem}

\begin{Remark}
 $\Dcal(X)$ and $\Pcal(X)$ are independent of the order of the elements of  $X$.
 The bases $\Bcal(X)$ and $\Bcyr(X)$ both depend on that order. 
 In Theorem~\ref{Theorem:MainTheoremCentral}, we assume that both bases are constructed 
 using the same order.
\end{Remark}

\begin{Example}
\label{Example:PandDbases}
 This is a continuation of Example~\ref{Example:MatrixBases}.
 See also Figure~\ref{Figure:ConvolutionCones}.
The elements of $\Bcyr(X)$ are
 \begin{align}
   R_{(x_1,x_2)}^{(x_1,x_2)} &= 1, \\
   R_{X}^{(x_1,x_3)} &= (T_{x_1} - T_{-x_1}) * ( T_{(x_2,x_3)} - T_{(-x_2,-x_3)} ) = s_2, \\
 \text{and }  R_{X}^{(x_2,x_3)} &= (T_{x_2} - T_{-x_2}) * ( T_{(x_1,x_3)} - T_{(-x_1,-x_3)} ) = s_1.
 \end{align}

The elements of $\Bcal(X)$ are
\begin{align}
   Q_{(x_1,x_2)} &= p_\emptyset = 1, \\
   Q_{(x_1,x_3)} &= p_{x_2} = t_2, \\
 \text{and }  Q_{(x_2,x_3)} &= p_{x_1} = t_1.
\end{align}
$\Bcal(X)$ and $\Bcyr(X)$ are obviously dual bases.
\end{Example}
The proof of Theorem~\ref{Theorem:MainTheoremCentral} is split into four lemmas.
Recall that for a basis $B=(b_1,\ldots, b_r)$  we defined
 a flag of subspaces $\{0\}= S_0^B \subsetneq S_1^B\subsetneq \ldots \subsetneq S_r^B = U\cong \R^r$, where
 $S_i^B:=\spa(b_1,\ldots, b_i)$.

\begin{Lemma}[Annihilation criterion]
\label{Lemma:AnnihilationCondition}
 Let $Z\subseteq U\cong \R^r$ be a finite list of vectors and let $B\subseteq Z$ be a basis.
Let $R^B_Z$ be the polynomial that is defined in \eqref{eq:BasisElementConvolution}.
 Let $C \subseteq Z$. Suppose  there exists $i\in [r]$ 
 \st $Z \cap ( S_{i}^B\setminus S_{i-1}^B) \subseteq C$.

 Then, $D_C R^B_Z = 0 $. 
\end{Lemma}

\begin{proof}
Note that $D_a T_{-a} = -\delta_0$. Using 
\eqref{eq:SplineDiff}, we obtain
 \begin{align}
 D_C R_i^B &= D_C ((-1)^{\abs{N_i}}T_{P_i} * T_{-N_i} - (-1)^{\abs{P_i}}T_{-P_i} * T_{N_i})
 \\&= 
     D_{C\setminus (S_i\setminus S_{i-1})}(   \delta_0  - \delta_0) = 0 .
\intertext{This implies}
 D_C R^B_Z &= D_{C\setminus (S_i\setminus S_{i-1})} R_1^B * \cdots * R_{i-1}^B * 0 * R_{i+1}^B * \cdots *R_r^B = 0.
 \notag \qedhere
\end{align}
\end{proof}

\begin{Lemma}[Inclusion]
\label{Lemma:ContainmentCentral}
 The polynomial  $R_{X\setminus E(B)}^B$ is contained  in $\Dcal(X\setminus E(B))$ for all $B\in \BB(X)$.
 Since $\Dcal(X\setminus E(B)) \subseteq \Dcal(X)$,
 this implies
 \begin{align}
   \Bcyr(X)\subseteq \Dcal(X).
 \end{align}
\end{Lemma}

\begin{proof}
%
%
%
%
%
Let $B\in \BB(X)$ and
let $C\subseteq X \setminus E(B)$ be a cocircuit, \ie 
 $C$ intersects all bases that can be selected from $X\setminus E(B)$.
  We need to show that $D_{C} R^B_{X\setminus E(B)} = 0$.
 $C$ can be written as $C=X\setminus (H \cup E(B))$  for some hyperplane $H\subseteq U$.

 Let $i$ be minimal \st $S_i\not\subseteq H$. Such an $i$ must exist since $ S_r=U $.
 Even $(S_i \setminus S_{i-1}) \cap H = \emptyset$ holds.
 This implies 
\begin{align}
  (X\setminus E(B)) \cap (S_i\setminus S_{i-1})\subseteq X\setminus (H \cup E(B))  = C.
 \end{align}
  By Lemma~\ref{Lemma:AnnihilationCondition}, this implies $D_C R^B_{X\setminus E(B)}=0$.
\end{proof}
 The following %
 lemma will be used only in the proof of  
 Lemma~\ref{Lemma:BaseDualityCentral}.
\begin{Lemma}
 \label{Lemma:DifferentBasesSameNumberOfEAelements}
 Let $B, D \in \BB(X)$. %
  Suppose that both bases are distinct but have the same number of externally active elements.

   Then there exists $i\in [r]$ %
 \st  
\begin{align}
\label{equation:TechnicalLemmaCentral}
( X \setminus  E(D) ) \cap   (S_i^{D} \setminus S_{i-1}^{D})
   \subseteq X\setminus (B\cup E(B)).
\end{align}
\end{Lemma}
\begin{proof}
 Let $B=(b_1,\ldots, b_r)$ and $D=(d_1,\ldots, d_r)$.
 Suppose that the lemma is false. Then there exist vectors $z_1,\ldots, z_r$ \st
\begin{align}
 z_i \in ( X \setminus  E(D) ) \cap   (S_i^{D} \setminus S_{i-1}^{D}) \cap (B\cup E(B)).
\end{align}
Those vectors form a basis because  $z_i \in S_i^{D} \setminus S_{i-1}^{D}$.
Since $z_i$ is not contained in $E(D)$,
$z_i\le d_i$ must hold.
This implies $E(D)\subseteq E(z_1,\ldots, z_r)$. 
Furthermore,  $E(z_1,\ldots, z_r)\subseteq E(B)$ since all $z_i$ are contained in 
 $B\cup E(B)$.

We have shown that $E(D)\subseteq E(B)$. This is a contradiction since 
 no finite set can be contained in a distinct set of the same cardinality.
\end{proof}

\begin{Lemma}[Duality]
\label{Lemma:BaseDualityCentral}
Let $B,D\in \BB(X)$. Let $Q_{B}= p_{X\setminus (B\cup E(B))} \in \Bcal(X)$ 
and let $R^{D}_{X\setminus E(D)}$ be the polynomial that is defined in 
\eqref{eq:BasisElementConvolution}.
 Then
\begin{align}
 \pair{Q_{B}}{R^{D}_{X\setminus E(D)}} = \frac{\delta_{B,D}}{\abs{\det(D)}}.
\end{align}
\end{Lemma}
$\delta_{B,D}$ denotes the Kronecker delta and 
 we consider $B$ and $D$ to be equal if 
there exist $1\le i_1 < \ldots < i_r \le N$ \st $B=(x_{i_1},\ldots, x_{i_r})=D$.

\begin{proof}
By Corollary~\ref{Corollary:HomPoly},
 $R^{D}_{ X\setminus E(D) }$ is a homogeneous
 polynomial of degree $N-r - \abs{E(D)}$.
 Thus, if $ \abs{E(B)} \neq \abs{E(D)}$, then
 $Q_{B}$ and $R^{D}_{ X\setminus E(D) }$ are homogeneous polynomials of different degrees
 and $\pair{Q_{B}}{R^{D}_{ X\setminus E(D) }} = 0$.

Now suppose that $B\neq D$ and both bases have the same number of externally active elements.
In this case, the statement follows from 
  Lemma~\ref{Lemma:AnnihilationCondition} and Lemma~\ref{Lemma:DifferentBasesSameNumberOfEAelements}.

The only case that remains is $B=D$.
 Recall that %
$R^{B}_{X\setminus E(B)}= R_1^B*\ldots * R_r^B$.
  Consider the $i$th factor $R_i^B$. The elements of $( X\setminus E(B))\cap(S_i \setminus S_{i-1})$ are used 
 for the construction of $R_i^B$. 
  Exactly one basis element is contained in this set: $b_i$. 
  Recall that in Section~\ref{Section:BasisElementsConstruction} we defined a partition $P_i \cup N_i = (X\setminus E(B))\cap(S_i \setminus S_{i-1})$.
  By construction, $b_i$ is positive, \ie $b_i\in P_i$.
 Now we apply the differential operator $D_{(P_i \setminus b_i) \cup N_i }$ to $R_i^B$:
 \begin{align}
   D_{(P_i\cup N_i) \setminus b_i} (  (-1)^{\abs{N_{i}}} \cdot T_{P_i} * T_{-N_i} 
   - (-1)^{\abs{P_i}} \cdot T_{-P_i} * T_{N_i} ) &= (T_{b_i} + T_{-b_i}).
 \end{align}
Now we can put things together. 
Note that $X\setminus (B \cup E(B)) = \bigcup_{i=1}^r ((P_i\setminus b_i) \cup N_i )$.
Hence,
 \begin{align}
   D_{X\setminus (B\cup E(B))} R^B_{X\setminus E(B)} &= (T_{b_1} + T_{-b_1}) * \cdots * (T_{b_r} + T_{-b_r}) = \frac{1}{\abs{\det(B)}}.
 \end{align}
This finishes the proof.
\end{proof}

\begin{proof}[Proof of Theorem~\ref{Theorem:MainTheoremCentral}]
 We know that $\Pcal(X)$ and $\Dcal(X)$ are dual via the pairing $\pair{\cdot}{\cdot}$
 and that $\Bcal(X)$ is a basis for $\Pcal(X)$.
  By Lemma~\ref{Lemma:BaseDualityCentral}, $\Bcyr(X)$ and $\Bcal(X)$ are dual to each other and
  by Lemma~\ref{Lemma:ContainmentCentral}, $\Bcyr(X)$ is contained in $\Dcal(X)$.
 Hence, $\Bcyr(X)$ is a basis for $\Dcal(X)$.
\end{proof}

\section{Deletion-contraction and exact sequences}
\label{Section:DelConExactSequencesCentral}
By 
Proposition~\ref{Proposition:CentralTutteEval}, the Hilbert series
of $\Dcal(X)$ and $\Pcal(X)$ are equal and an evaluation of the Tutte polynomial.
 In particular, they
  satisfy a deletion-contraction identity that
extends in a natural way to our algebraic setting.
 This is reflected by two dual short exact sequences. 

 In this section we define deletion and contraction and we explain those two exact sequences. 
 While the two sequences were known 
 before, their duality has not yet been stated explicitly in the literature.

\smallskip
Two important matroid operations are deletion and contraction.
 For realisations of matroids, they are defined as follows. Let $X\subseteq U$ be a finite 
 list of vectors and let $x\in X$.
The \emph{deletion} of $x$ is the list $X\setminus x$. 
The \emph{contraction} of $x$ is the list $X/x$, which is defined to be the 
image of $X\setminus x$ under the
 projection $\pi_x : U\to U/x$.
 The space $U$ and other terminology used here are defined in Subsection~\ref{Subsection:Notation}.

 The space $\Pcal(X/x)$ is contained in the symmetric algebra $\sym(U/x)$. 
 If  $x=s_r$, then there is a natural isomorphism $\sym(U/x)\cong  \R[s_1,\ldots, s_{r-1}]$
 that maps $\bar s_i$ to $s_i$. This 
 isomorphism depends on the choice of the basis $(s_1, \ldots, s_r)$ for $U$.
 Under this identification,  $\sym(\pi_x)$ is the map from 
 $\R[s_1,\ldots, s_{r}]$ to $\R[s_1,\ldots, s_{r-1}]$ that sends $s_r$ to zero and $s_1,\ldots, s_{r-1}$ to themselves.

 For $\Dcal(X /x )$, the situation is simpler: this space is contained in
 $\sym( (U/x)^*) \cong \sym(x^\ann)$, where $x^\ann\subseteq V$ denotes the annihilator of $x$.
 $\sym(x^\ann)$ is a subspace of $\sym(V)$. We denote the inclusion map by $j_x$.
 If $x=s_r$, then $\sym( (U/x)^*)$ is isomorphic to
 $\R[t_1,\ldots, t_{r-1}]$. This is a canonical isomorphism.

%
For a graded vector space $S$, we write
$S[1]$ for the  vector space  with the degree shifted up by one.
\begin{Proposition}[\cite{ardila-postnikov-2009%
}]
\label{Theorem:ExactSequencesCentralP}
Let $X\subseteq U\cong \R^r$ be a finite list of vectors 
  that spans $U$  and let $x\in X$ be neither a loop nor a coloop.
Then, the following sequence of graded vector spaces is exact:
  \begin{align}
   0 \to \Pcal(X\setminus x)[1] \stackrel{\cdot p_x}{\to} \Pcal(X) 
   \stackrel{\sym(\pi_x)}{\longrightarrow} \Pcal(X / x) \to 0. 
  \end{align}
\end{Proposition}

\begin{Proposition}[\cite{boor-ron-shen-1996-I}]
\label{Theorem:ExactSequencesCentralD}
Let $X\subseteq U\cong \R^r$ be a finite list
  of vectors  that spans $U$  and let $x\in X$ be neither a loop nor a coloop.
Then, the following sequence of graded vector spaces is exact:
 \begin{align}
\label{Equation:DexactSequenceCentral}
  0 \to \Dcal(X / x) \stackrel{j_x}{\to} \Dcal(X) \stackrel{D_x}\to \Dcal(X\setminus x)[1] \to 0. 
 \end{align}
\end{Proposition}
Note that  \eqref{Equation:DexactSequenceCentral}
is a special case of
 (1.12) in \cite{boor-ron-shen-1996-I} and it is exact by the results in that paper.
\begin{Remark}
\label{Remark:PDsequenceDuality}
Proposition~\ref{Theorem:ExactSequencesCentralP} and 
Proposition~\ref{Theorem:ExactSequencesCentralD} are equivalent
because of the duality of $\Pcal(X)$ and $\Dcal(X)$.
\end{Remark}

\begin{proof}[Proof of Remark~\ref{Remark:PDsequenceDuality}]
We only show that Proposition~\ref{Theorem:ExactSequencesCentralD} implies Proposition~\ref{Theorem:ExactSequencesCentralP}.
The other implication is similar.

 Since dualisation of finite dimensional vector spaces is a contravariant exact functor,  the following
 sequence is exact by Proposition~\ref{Theorem:ExactSequencesCentralD}:
  \begin{align}
   0 \to \Dcal(X\setminus x)^* \stackrel{(D_x)^*}{\longrightarrow} \Dcal(X)^* 
   \stackrel{(j_x)^*}{\longrightarrow} \Dcal(X / x)^* \to 0. 
  \end{align}
By Proposition~\ref{Proposition:PDduality},
$\Pcal(X)$ is isomorphic to $\Dcal(X)^*$ via $q\mapsto \pair{q}{\cdot}$.
 Hence, it is sufficient to show that 
 the following two diagrams commute:
\begin{equation}
\xymatrix{
     \Pcal(X\setminus x) \ar[r]^{q \mapsto \pair{q}{\cdot}} \ar[d]^{\cdot p_x}  & \ar[d]^{(D_x)^*} \Dcal(X\setminus x)^* \\
     \Pcal(X) \ar[r]^{q \mapsto \pair{q}{\cdot}} &  \Dcal(X)^*
  }
  \qquad  \raisebox{-8mm}{\text{and}}  \qquad
   \xymatrix{
     \Pcal(X) \ar[r]^{q \mapsto \pair{q}{\cdot}} \ar[d]^{\sym(\pi_x)} & \ar[d]^{(j_x)^*} \Dcal(X)^* \\
     \Pcal(X/x) \ar[r]^{q \mapsto \pair{q}{\cdot}} &  \Dcal(X/x)^*
  } 
 \;\raisebox{-8mm}{.} 
\end{equation}
For the diagram on the left, we have to show that
$\pair{p_x q}{\cdot}= \pair{q}{D_x\: \cdot\,} $ for all
 $q\in\Pcal(X\setminus x)$.
This is easy. 

For the diagram on the right, 
we have to show that
$\pair{\sym(\pi_x) q}{\cdot}= \pair{q}{j_x(\cdot)}$
 for all $q\in\Pcal(X)$. 
If we choose a basis with $s_r=x$ this
 follows  from the fact that 
 $\diff{t_r}f=0$ for all $f\in\R[t_1,\ldots, t_{r-1}]$.
\end{proof}

\section{Forward exchange matroids}
\label{Section:Combinatorics}
In this section we introduce forward exchange matroids.
A forward exchange matroid is an ordered matroid
together with a subset of its set of bases that satisfies  a weak version of the
 basis exchange axiom
 \eqref{equation:ExchangeProperty}.

The motivation for this definition is the following:
 the author noticed that most of the results in Subsection~\ref{Subsection:CentralZA}
 and Section~\ref{Section:MainCentral} hold in a far more general context.
 An important ingredient of the definitions of the spaces $\Pcal(X)$ and $\Dcal(X)$ 
 and their bases 
 is the set of bases $\BB(X)$ of the list $X$.
 These two spaces still have 
 nice properties if we modify their definitions
 and use only a suitable subset $\BB'$ of $\BB(X)$.
 It turned out that forward exchange matroids are the right axiomatisation 
 of ``suitable subset''.

\smallskip

Let $(A,\BB)$ be an ordered  matroid of rank $r$ and
let $B=(b_1,\ldots, b_r)\in \BB$ be an ordered basis.
The flag 
 \eqref{eq:flagofsubspaces} can be defined in combinatorial terms:
for $i\in\{ 0,\ldots, r\}$, we define $S_i = S_i^B := \clos \{ b_1,\ldots, b_i \}\subseteq A$. 
Hence, we obtain a flag of flats 
 \begin{align}
 \label{eq:flagofsubspacesMatroidal}
   \{ x\in A : x \text{ loop}\} = S_0^B  \subsetneq S_1^B \subsetneq S_2^B \subsetneq \ldots \subsetneq S_r^B = A.
 \end{align}
One can easily show that for a basis $B\in\BB$ and $i\in [r]$, the following statement holds:
\begin{align}
\label{eq:BasisForwardBackwardExchange}
 \text{let } x\in S^B_i \setminus S_{i-1}^B. \text{ Then } B'= (B\setminus b_i) \cup x \text{ is also in } \BB.
\end{align}
Note that 
$x\in S^B_i \setminus S_{i-1}^B$
 satisfies $x>b_i$ if and only if $x$ is externally active with respect to $B$. This motivates the name
 of the following definition.

\begin{Definition}[Forward exchange property]
Let $(A,\BB)$ be an ordered matroid and let $\BB'\subseteq \BB$.
We say that the set of  bases $\BB'$ has the \emph{forward exchange property}
 if the following  holds for all bases $B\in \BB'$ and all $i\in [r]$:  
\begin{align}
\label{eq:xInSi}
\text{let } x\in S^B_i \setminus (S_{i-1}^B \cup E(B)). \text{ Then } B'= (B\setminus b_i) \cup x  \text{ is also in }
  \BB'.
\end{align}
\end{Definition}

\begin{Remark}
Note that
$S_j^B=S_j^{B'}$ holds for all $j\ge i$. The vector $x$ is not necessarily the $i$th vector of $B'$. However, if this is the case,
then $S_j^B=S_j^{B'}$ holds for all $j \in [r]$,
\ie $B$ and $B'$ define the same flag. 
\end{Remark}

\begin{Definition}[Forward exchange matroid]
 A triple 
$(A,\BB, \BB')$ is called a \emph{forward exchange matroid} if
 $(A,\BB)$ is an ordered matroid and $\BB'$ is a subset
 of the set of bases $\BB$ with the forward exchange property.
\end{Definition}

\begin{Remark}
 In this paper, we mainly consider realisations of forward exchange matroids, \ie
  pairs $(X, \BB')$ where $X$ is a list of vectors and
  $\BB'\subseteq \BB(X)$ is a set of bases with the forward exchange property.
 \end{Remark}

\begin{Definition}[Tutte polynomial of forward exchange matroids]
\label{Definition:ForwardExchangeTutte}
Let $(A,\BB,\BB')$ be a forward exchange matroid. 
 We define its \emph{Tutte polynomial} by
\begin{align}
  T_{(A,\BB,\BB')}(x,y) := \sum_{B\in \BB'} x^{\abs{I(B)}} y^{\abs{E(B)}},
\end{align}
 where
 $I(B)$ and $E(B)$ denote the sets of internally and externally active elements
 with respect to $B$ in the ordered matroid $(A,\BB)$.
\end{Definition}

\begin{Remark}
It would be interesting to clarify the relationship between forward-exchange matroids 
 and other set systems studied in combinatorics  such as greedoids
 \cite{bjoerner-ziegler-1992,korte-lovasz-1981} or Jos\'e Samper's various generalisations of matroid complexes \cite{samper-fpsac-toappear-2016}.
\end{Remark}

\section{Generalised $\Dcal$-spaces and $\Pcal$-spaces}
\label{Section:GeneralZonotopalSpaces}
Earlier, we considered the spaces 
 $\Dcal(X)$ and $\Pcal(X)$ for a given list of vectors $X$. 
 The construction of these spaces relied mainly on the matroidal properties of the list $X$, namely 
 on the sets of bases and  cocircuits. 

Motivated by questions in approximation theory, various authors 
 generalised these constructions.
 Given a list $X$ and a subset $\BB'$ of its set of bases $\BB(X)$, one can define
  a set $\Dcal(X,\BB')$ as the kernel of the ideal generated by the $\BB'$-cocircuits 
 (\ie sets that intersect all bases in $\BB'$). 
Under certain conditions, $\dim \Dcal(X,\BB') = \abs{\BB'}$ still holds. 
  In this section we show that if the set $\BB'$ has the forward exchange property, 
 this equality holds 
  and there is a canonical dual space $\Pcal(X,\BB')$.
 Both, the generalised $\Dcal$-spaces and the generalised $\Pcal$-spaces satisfy deletion-contraction identities as in
 Section~\ref{Section:DelConExactSequencesCentral} and there are canonical 
 bases for both spaces that are dual.

\subsection{Definitions and Main Result}

\begin{Definition}[generalised $\Dcal$-spaces]
 \XBBintro
  A set $C\subseteq X$ is called a \emph{$\BB'$-cocircuit} if $C$ intersects every basis in $\BB'$ and $C$ is 
 inclusion-minimal
 with this property.

  The \emph{generalised $\Dcal$-space} defined by $X$ and $\BB'$ is
  \begin{align}
     \Dcal(X,\BB') &:=  \{  f : D_C f = 0 \text{ for all $\BB'$-cocircuits }  C  \}
               =  \ker \Jcal(X,\BB'), \\
\text{where } \Jcal(X,\BB') &:= \ideal\{  p_C : C\subseteq X \text{ is a }\BB'\text{-cocircuit} \}.
  \end{align}
\end{Definition}

\begin{Proposition}[{%
\cite[Theorem 6.6]{boor-ron-1991}%
}]
\label{Prop:DdimSimpleBound}
\XBBintro Then
 \begin{align}
  \dim \Dcal(X,\BB') \ge \abs{\BB'}.
  \label{eq:DspaceDimInequalitySimple}
 \end{align}
\end{Proposition}

\begin{Definition}[generalised $\Pcal$-spaces]
\XBBintro
Then we define 
\begin{align}
   \Bcal(X,\BB') &:=  \{ Q_B : B\in \BB'  \} = \{ p_{X \setminus ( B \cup E(B) ) } : B \in \BB' \} \\
  \text{and } \Pcal(X,\BB') &:= \spa \Bcal(X,\BB'). 
\end{align}
We call $\Pcal(X,\BB')$ the  \emph{generalised $\Pcal$-space} defined by $X$ and $\BB'$.
\end{Definition}
\begin{Remark}
 The set $\Bcal(X,\BB')$ is a basis for $\Pcal(X,\BB')$. By definition, it is spanning 
 and it is linearly independent because it is a subset of $\Bcal(X)$.
\end{Remark}
If the set $\BB'$ has the forward exchange property, 
the spaces $\Dcal(X,\BB')$ and $\Pcal(X,\BB')$ have many nice properties. 
Here  is the Main Theorem of this section.
\begin{Theorem}
\label{Theorem:MainTheoremGeneral}
\XBBintroFEP

Then,
 the  generalised $\Dcal$-space $\Dcal(X,\BB')$ and the generalised $\Pcal$-space $\Pcal(X,\BB')$ 
 are dual via the pairing $\pair{\cdot}{\cdot}$.
In addition,
\begin{align}
 \Bcyr(X,\BB') :=  \{ \abs{\det(B)} R^B_{X\setminus E(B)} : B\in \BB'   \}
\end{align}
 is a basis for %
$\Dcal(X,\BB')$ and this 
 basis is dual to the basis $\Bcal(X,\BB')$ for %
  $\Pcal(X,\BB')$.
\end{Theorem}
\begin{Corollary}
\XBBintroFEPcard
Then
\begin{align*}
 \hilb(\Pcal(X,\BB'),q) = \hilb(\Dcal(X,\BB'),q) &= \sum_{B\in \BB'} q^{N-r-\abs{E(B)}} 
 \\
 &= q^{N-r} T_{(X,\BB(X),\BB')}(1,\frac 1q).
\end{align*}
\end{Corollary}
Here are two examples that help to understand generalised $\Dcal$-spaces, generalised $\Pcal$-spaces, and  
Theorem~\ref{Theorem:MainTheoremGeneral}.
\begin{Example}
\label{Example:GeneralSpaces}
Let $X=(e_1,e_2,e_3,a,b)\subseteq \R^3$ where $e_1,e_2,e_3$ 
 denote the unit vectors and $a =(\alpha, \beta,\gamma)$ and $b$ are generic.  
 In particular,  $\alpha,\beta,\gamma\neq 0$.
Let  
\begin{equation*}
\BB':=\{(e_1e_2e_3),\,(e_1e_2a),\,(e_1e_2b),\, (e_1e_3a),\, (e_1e_3b),\, (e_2e_3a),\, (e_2e_3b)\}\subseteq\BB(X). 
\end{equation*}
The reader is invited to check that $\BB'$ has the forward exchange property.
 The Tutte polynomial is 
 $T_{(X,\BB(X),\BB')}(x,y) =   
 3x + 3y + y^2 $ and
the $\BB'$-cocircuits are  $\{e_1e_2,\, e_1e_3,\, e_2e_3,\, e_1ab,\, e_2ab,\, e_3ab \}$.
Hence,
\begin{align*}
 \Dcal(X,\BB') &= 
   \ker\ideal\{ s_1s_2,\, s_1s_3,\, s_2s_3,\, s_1 p_{ab},\, s_2 p_{ab},\, s_3 p_{ab} \} 
   \\ &= \spa \{ 1,\,t_1,\,t_2,\,t_3,\, t_1^2,\, t_2^2,\, t_3^2 \}, \\
 \Bcal(X,\BB') &=   \{ 1,\, p_{e_3},\, p_{e_3a},\, p_{e_2},\, p_{e_2a},\, p_{e_1},\, p_{e_1a}  \}, \\
 \Pcal(X,\BB') &= \spa \{1,\, s_1,\, s_2,\, s_3,\: 
                    s_1(\alpha s_1 + \beta s_2 + \gamma s_3) ,\, 
                    s_2(\alpha s_1 + \beta s_2 + \gamma s_3) , \\
	      &\qquad\qquad\qquad\qquad\qquad\qquad\qquad\qquad\qquad
		                    s_3(\alpha s_1 + \beta s_2 + \gamma s_3) \}, \text{ and} \\
 \Bcyr(X,\BB') &= \left\{ 1,\, t_3,\, \frac{t_3^2}{2\gamma},\, t_2,\, \frac{t_2^2}{2\beta},\, t_1,\, 
                    \frac{t_1^2}{2\alpha} \right\}.          
\end{align*}

\end{Example}
\begin{Example}
\label{Example:PnotInverseSystem}
 Let $N\ge 3$ be an integer. Let $X_N=(x_1,\ldots, x_N)$ be a list of vectors in general position in $\R^2$
 with $x_1=e_1$, $x_2=e_2$, and
 $x_3=e_1+e_2$. 
 In addition, we suppose that the second coordinate of all vectors
 $x_i$ ($i\ge 3$) is one.
 Let  $\BB' := \{ (x_1,x_i) : i\in \{ 2,\ldots, N\}\} \cup \{ (x_2,x_3) \}$.
 Note that $\BB'$ is 
 unimodular, \ie all elements have determinant $1$ or $-1$ and $\BB'$ has the forward exchange
 property. 
 Then,
 \begin{align}
   \Dcal(X_N,\BB')  &= \ker\ideal \{ p_{x_1x_2}, p_{x_1x_3}, p_{x_2\cdots x_N} \} = \spa \{ 1, t_1,t_2, t_2^2, t_2^3, 
		  \ldots, t_2^{N-2} \},  \nonumber \\
   \Bcyr(X_N,\BB') &= \left\{ 1,t_1,t_2, \frac{t_2^2}{2}, \ldots, \frac{t_2^{N-2}}{(N-2)!}   \right\},\text{ and} \\
   \Bcal(X_N,\BB')  &= \{1, p_{x_1} \} \cup \{   p_{x_2\cdots x_i } : i\in \{ 2,\ldots, N-1  \} \}.
 \end{align}
 \end{Example}

Now we embark on the proof of
Theorem~\ref{Theorem:MainTheoremGeneral}.
We start with the following simple lemma.
\begin{Lemma}[Inclusion]
\label{Lemma:ContainmentGeneral}
 \XBBintroFEP
Then
\begin{align}
  \Bcyr(X,\BB')\subseteq \Dcal(X,\BB').
\end{align}
\end{Lemma}
\begin{proof}%
 Let $B\in \BB'$ and let $\abs{\det(B)} R^B_{X\setminus E(B)}\in\Bcyr(X,\BB')$ be the corresponding basis element.
 Let $C \subseteq X$ be a $\BB'$-cocircuit, \ie an inclusion-minimal subset of $X$ that intersects every basis in $\BB'$.

Let $B=(b_1,\ldots, b_r)$.
 If there exists an $i$ \st $(S_i^B\setminus S_{i-1}^B) \cap (X \setminus E(B)) \subseteq C$, 
  we are done by Lemma~\ref{Lemma:AnnihilationCondition}.
 Now suppose that this is not the case, \ie for every $i\in[r]$, there
 is a $z_i \in (S_i^B\setminus S_{i-1}^B) \cap (X \setminus (E(B) \cup C) )$.
 Then we define a sequence of bases $B_0,\ldots, B_r$ by
\begin{align}
 B_0 := B \quad\text{and}\quad B_{i} := (B_{i-1} \setminus b_i) \cup z_{i} \text{ for } i\in [r]. 
\end{align}
 The lists $B_i$ are indeed bases and even though in general, they might define different flags, they satisfy
\begin{align} 
 (S_{i}^{B_{i-1}} \setminus S_{i-1}^{B_{i-1}}) \cap (X\setminus E(B_{i-1}))
     = (S_{i}^{B} \setminus S_{i-1}^{B}) \cap (X\setminus E(B)),  
\end{align}
because $\spa(b_1,\ldots, b_i)=\spa(z_1,\ldots, z_i)$ for all $i\in [r]$. Hence, $B_i\in \BB'$ implies 
 $B_{i+1}\in \BB'$
because 
 $\BB'$ has the forward exchange property.
 In particular, $B_r=(z_1,\ldots, z_r)\in \BB'$. By construction, $B_r\cap C= \emptyset$. 
This is a contradiction.
\end{proof}
\begin{Definition}
Let $(A,\BB)$ be a matroid and let $\BB'\subseteq \BB$. %
Let $x\in A$.
$\BB'$ can be partitioned as $\BB'=\BB_{\setminus x} \cup \BB_{| x}$, where
\begin{align}
\BB'_{\setminus x}&:=\{B\in \BB' : x\not\in B\} \text{ denotes the \emph{deletion} of $x$ and}\\
\BB'_{|x}
&:=\{ B\in \BB' : x\in B\} \text{ the \emph{restriction} to $x$.}
\end{align}
If we are given a list of vectors $X\subseteq U$ and a set of bases $\BB'\subseteq \BB(X)$, we can also define the 
 \emph{contraction} $\BB'_{ /x}$.
Recall that $\pi_x : U\to U/x$ denotes the canonical projection.
 Then, we define
\begin{align}
 \BB'_{/x } := \{ \pi_x(B\setminus x) : x\in B \in \BB'_{|x}\}.
\end{align}
\end{Definition}

\begin{Remark}
 For technical reasons, it is helpful to distinguish the contraction $\BB'_{/x}$ and 
 the restriction $\BB'_{|x}$ although there is a canonical bijection between both sets.
\end{Remark}

We now introduce the concept of placibility.
 This is a condition on a set of bases $\BB'$ which implies equality in 
\eqref{eq:DspaceDimInequalitySimple}. 
\begin{Definition}[{\cite{boor-ron-shen-1996-I}, see also \cite{li-ron-2013}}]
Let $(A,\BB)$ be a matroid and let $\BB'\subseteq \BB$ be a non-empty set of bases.
\begin{enumerate}[(i)]
 \item 
We call an element $x\in A$  \emph{placeable} in
$\BB'$ if for each $B\in \BB'$, there exists an
element $b\in B$ such that $ ( B \setminus b) \cup x \in \BB'$.
 \item 
 We say that $\BB'$ is \emph{placible}
if one of the following two conditions holds:
\begin{enumerate}
\item $\BB'$ is a singleton or
\item there exists $x\in A$ \st $x$ is placeable in $\BB'$ and 
 both, $\BB'_{|x}$ and $\BB'_{\setminus x}$ are non-empty
and placible.
\end{enumerate}
\end{enumerate}
\end{Definition}
\begin{Proposition}[\cite{boor-ron-shen-1996-I}]
\label{Proposition:PlacibleimpliesDdimGood}
 \XBBintro
If $\BB'$ is 
placible, then $\dim \Dcal(X,\BB')= \abs{\BB'}$.
\end{Proposition}
\begin{Lemma}%
\label{Lemma:SFPimpliesPlacible}
Let $(A,\BB,\BB')$ be a forward exchange matroid.
Then 
$\BB'$ is  placible.
\end{Lemma}
\begin{proof}
If $\abs{\BB'}=1$, then $\BB'$ is placible by definition. 
Now let $\abs{\BB'} \ge 2$. 
Let 
 $x$ be the minimal element in $A$ \st both, $\BB'_{|x}$ and $\BB'_{\setminus x}$ are non-empty.
 Such an element must exist if $\abs{\BB'} \ge 2$. 

 We now show  that $x$ is placeable in $\BB'$.
 Let $B=(b_1,\ldots, b_r)$ be a basis in $\BB'$ and let $i\in [r]$ \st $x\in S_i^B\setminus S_{i-1}^B$.
 We claim that $x\le b_i$. Suppose it is not. Because of the minimality of $x$, 
 this implies that $b_1,\ldots, b_i$ are contained in all bases in $\BB'$.
 Since $x\in\spa(b_1,\ldots, b_i)$, this implies that $x$ is not contained in any basis. This is a contradiction
 because we assumed that $\BB'_{|x}$ is non-empty.

Now we have established that $x\le b_i$.
 This implies that $x$ is not externally active. 
Hence, because of the forward exchange property, $ ( B \setminus b_i ) \cup x  \in \BB'$, \ie $x$ is placeable in $B\in\BB'$.

It remains to be shown that 
 $\BB'_{\setminus x}$ and $\BB'_{/x}$ are both placible.
 By induction, it is sufficient to show that both sets have the forward exchange property.
 For $\BB'_{\setminus x}$, this is clear.
 For $\BB'_{|x}$, this follows from the following  fact: 
 by the choice of $x$, 
all $a\in A$ that satisfy $a < x$ are either contained in all bases or in no basis in  $ \BB'_{|x}$.
\end{proof}
\begin{proof}[Proof of Theorem~\ref{Theorem:MainTheoremGeneral}]
By Lemma~\ref{Lemma:ContainmentGeneral}, $\Bcyr(X,\BB')\subseteq \Dcal(X,{\BB'})$ holds.
By Proposition~\ref{Proposition:PlacibleimpliesDdimGood} and by Lemma~\ref{Lemma:SFPimpliesPlacible}, 
 $\dim \Dcal(X,{\BB'}) = \abs{\BB'} = \abs{\Bcyr(X,\BB')}$. 
Linear independence of $\Bcyr(X,\BB')$ 
 is clear because it is a subset of $\Bcyr(X)$.  
For the same reason, the duality with $\Bcal(X,\BB')\subseteq \Bcal(X)$ follows from
Lemma~\ref{Lemma:BaseDualityCentral}.
\end{proof}

\begin{Remark}
 The correspondence between $\Dcal(X)$ and the set of vertices $S$ of a  
 hyperplane arrangement $\Hcal(X,c)$ in general position that is stated in 
 Theorem~\ref{Theorem:HyperplaneArrangementD} generalises in a straightforward way
 to a correspondence between $\Dcal(X,\BB')$ and
 the subset of $S$ that is defined by $\BB'$.
\end{Remark}

\subsection{Deletion-contraction and exact sequences} 
In this subsection, we show that the  results in 
 Section~\ref{Section:DelConExactSequencesCentral} about deletion-contraction and exact sequences 
 naturally extend to generalised $\Dcal$-spaces and $\Pcal$-spaces.
 We use the same terminology as in that section.

Recall that
for a graded vector space $S$, we write
$S[1]$ for the vector space  with the degree shifted up by one.
Let $(A,\BB,\BB')$ be a forward exchange matroid. If $\abs{\BB'} \ge 2$,
there is an element $x\in A$ that is neither a loop nor a coloop, \ie 
$\BB'_{|x}$ and $\BB'_{\setminus x}$ are both non-empty.
\begin{Proposition}
\label{Prop:ExactSequenceDgeneralCondition}
\XBBintroFEP
Let $x$ be the minimal element of $X$ that is neither a loop nor a coloop.
Then, the following sequence of graded vector spaces is exact:%
 \begin{align}
  0 \to \Dcal(X / x, \BB'_{/x}) \stackrel{j_x}{\to} \Dcal(X,\BB') 
       \stackrel{D_x}\to \Dcal(X\setminus x,\BB'_{\setminus x})[1] \to 0. 
 \end{align}
\end{Proposition}

\begin{proof}
Let $B\in \BB'$ be a basis that does not contain $x$. 
 Because of the minimality, $x$ is not externally active with respect to $B$.
This implies
$D_x R^B_{X\setminus E(B)} = R^B_{X\setminus (E(B)\cup x)}$. 
Hence, $D_x  : \{ \abs{ \det(B) } R^B \in\Bcyr(X,\BB') : x\not\in B \}  \to \Bcyr(X \setminus x, \BB'_{\setminus x})$ is a bijection
and consequently, $D_x$ maps $\Dcal(X,\BB')$ surjectively to $\Dcal(X\setminus x,\BB'_{\setminus x})$.
For the rest of the proof, we refer the reader to \cite{boor-ron-shen-1996-I}, in particular
to the explanations following (1.12)
 and to Theorem 2.16.
\end{proof}

\begin{Proposition}
\label{Prop:ExactSequencePgeneralCondition}
\XBBintroFEP
Let $x$ be the minimal element of $X$ that is neither a loop nor a coloop.
Then, the following sequence of graded vector spaces is exact:
   \begin{align}
    \label{equation:sequenceGeneralP}
   0 \to \Pcal(X\setminus x,\BB'_{\setminus x})[1] \stackrel{\cdot p_x}{\to} \Pcal(X,\BB') 
      \stackrel{\sym(\pi_x)}{\longrightarrow} \Pcal(X/x,\BB'_{/x}) \to 0.  
  \end{align}
\end{Proposition}

\begin{proof}
One can easily check this 
  for the bases of the $\Pcal$-spaces.
Alternatively, it can be deduced from Proposition~\ref{Prop:ExactSequenceDgeneralCondition} using a duality argument
 as in the proof of Remark~\ref{Remark:PDsequenceDuality}.
\end{proof}

\begin{Remark}
The exact sequences in this section require $x$ to be minimal in contrast to the ones
 Section~\ref{Section:DelConExactSequencesCentral}, where $x$ can be any element that is neither a loop nor a coloop.
This reflects the fact that matroids have  an (unordered) ground set, while forward exchange
 matroids have an (ordered) ground list.
\end{Remark}

\begin{Remark}
One could replace $\Dcal(X/x,\BB'_{/x})$ by $\Dcal(X,\BB'_{|x})$ in 
 Proposition~\ref{Prop:ExactSequenceDgeneralCondition}.
 The analogous replacement  in 
 Proposition~\ref{Prop:ExactSequencePgeneralCondition}
 would be problematic.
  The reasons for that are explained in Section~\ref{Section:DelConExactSequencesCentral}.
\end{Remark}

\subsection{$\Pcal(X,\BB')$ as the kernel of a power ideal}
By now, we have seen that most of the results
 regarding  $\Dcal(X)$ and $\Pcal(X)$ that we stated earlier also hold for the generalised $\Dcal$-spaces and $\Pcal$-spaces.
The  only thing that is missing is a power ideal
$\Ical(X,\BB')$ \st 
 $\Pcal(X,\BB') = \ker \Ical(X,\BB')$. 
Unfortunately, such an ideal does not always exist.

In this section, we describe the natural
 candidate for this power ideal and we give an example where 
 its kernel is equal to $\Pcal(X,\BB')$ and one where it is not.
\begin{Definition}
 \XBBintro
 Recall that $V:=U^*$. 
We define a function $\kappa : V \to \N$ by
\begin{align}
 \kappa(\eta) &:= \max_{B\in\BB'} \abs{X\setminus (B\cup E(B)\cup \eta^\ann)} \\
 \text{and }\; \Ical(X,\BB') &:=  \ideal \{ p_\eta^{\kappa(\eta) +1} : \eta\in V\setminus  \{ 0 \}  \}.
\end{align}
\end{Definition}
\begin{Lemma}
\label{Lemma:ContainmentP}
 \XBBintro
Then
\begin{align}
\Pcal(X,\BB') \subseteq \ker\Ical(X,\BB').
\end{align}
\end{Lemma}
\begin{proof}
It is sufficient to show that all elements of the basis $\Bcal(X,\BB')$ are contained in $\ker\Ical(X,\BB')$.
 Let $B\in \BB'$ and  let $\eta \in V \setminus \{0\} $.
Then
 \begin{align}
   D_\eta^{\kappa(\eta)} p_{X\setminus (B\cup E(B)) } 
   &=
   p_{(X \cap \eta^\ann) \setminus ( B\cup E(B))} D_\eta^{\kappa(\eta) + 1 } p_{X\setminus (B\cup E(B) \cup \eta^\ann) } 
   = 0
 \end{align}
   The first equality follows from Leibniz's law.
   The second  follows from the fact that by definition, $\kappa(\eta) \ge \abs{X\setminus (B\cup E(B) \cup \eta^\ann)} + 1$.
\end{proof}
\begin{Remark}
If one examines the proof of Lemma~\ref{Lemma:ContainmentP}, 
 one immediately sees that $\Ical(X,\BB)$ is the only power ideal for which
 $\Pcal(X,\BB') = \ker\Ical(X,\BB')$ can possibly hold.
\end{Remark}
\begin{Remark}
\label{Remark}
In some
cases, $\Pcal(X,\BB')$ and $\ker\Ical(X,\BB')$ are equal (see Example~\ref{Example:GeneralSpacesIideal}). 
In other cases however, $\Pcal(X,\BB')$ is not even closed
 under differentiation (see Example~\ref{Example:PnotInverseSystem}). 
\end{Remark}
Remark~\ref{Remark}
naturally leads to the following question.
\begin{Question}
 Is there a simple criterion to decide 
 whether $\Pcal(X,\BB')$ is closed under differentiation or 
 if $\Pcal(X,\BB') = \ker\Ical(X,\BB')$ holds?
\end{Question}
\begin{Example}
\label{Example:GeneralSpacesIideal}
This is a continuation of Example \ref{Example:GeneralSpaces}.
Recall that we considered 
the list $X=(e_1,e_2,e_3,a,b)\subseteq \R^3$ where $a$ and $b$ are generic vectors and 
$a=(\alpha, \beta,\gamma)$ with $\alpha,\beta,\gamma\neq 0$.
The set of bases is
\begin{align*}
 \BB'=\{(e_1e_2e_3),\,(e_1e_2a),\,(e_1e_2b),\, (e_1e_3a),\, (e_1e_3b),\, (e_2e_3a),\, (e_2e_3b)\}\subseteq\BB(X).
\end{align*}
In order to calculate the function $\kappa$, we first determine  
 the inclusion-maximal lists in $\{ X\setminus (B\cup E(B)) : B\in \BB\}$.
 Those are $(e_1a)$, $(e_2a)$, and $(e_3a)$. We can deduce that $\kappa(\eta)$ is 
 one if $\eta \in a^\ann$ and two otherwise. We obtain
\begin{align*}
 \Ical(X,\BB') &= \ideal\{  p_{(\alpha,-\beta,0)}^2, p_{(0,\beta,-\gamma)}^2, p_{(\alpha,0,-\gamma)}^2  \} + \R[s_1,s_2,s_3]_{\ge 3}
\text{ and}
\\
 \Pcal(X,\BB') &=  \ker\Ical(X,\BB') =\spa \{1, s_1, s_2, s_3,\, 
                    s_1(\alpha s_1 + \beta s_2 + \gamma s_3) ,\, 
                    \\
	      &\qquad\qquad\qquad\qquad\qquad\quad
       s_2(\alpha s_1 + \beta s_2 + \gamma s_3),
		                    s_3(\alpha s_1 + \beta s_2 + \gamma s_3) \}.
\end{align*}
 The degree two component of $\R[s_1,s_2,s_3]/\Ical(X,\BB')$ is three-dimensional.
 This implies that $\Pcal(X,\BB')=\ker
\Ical(X,\BB')$.
\end{Example}
\section{Comparison with previously known zonotopal spaces}
\label{Section:ZonotopalAlgebra}
 In this section we review the
 definitions of various zonotopal spaces that have been studied previously by other authors.
 It turns out that they are all special 
 cases 
 of the generalised $\Dcal$-spaces and $\Pcal$-spaces 
 that we introduced in Section~\ref{Section:GeneralZonotopalSpaces}.
 The most prominent examples are of course the central spaces $\Dcal(X)$ 
 and $\Pcal(X)$ that we obtain if we choose $\BB'=\BB(X)$.

\smallskip
 Let  $A = (a_1,\ldots, a_n) \subseteq \R^{r}$ be a  list of vectors that spans $\R^r$ and let
 $X=(x_1,\ldots, x_N) \subseteq  \R^{r}$, where $N\ge n$ and $a_i=x_i$ for $i\in [n]$.
 In \cite{holtz-ron-2011, holtz-ron-xu-2012, li-ron-2013}, 
 the spaces $\Dcal(X,\BB')$ and $\Pcal(X,\BB')$  where studied for 
 certain sets of bases $\BB' \subseteq \BB(X)$.
 Here are the definitions of these sets of bases. 
\begin{Definition}[Internal and external bases \cite{holtz-ron-2011}]
 \label{Definition:InternalExternalBases}
 Let $A\subseteq \R^r$ be a list of vectors that spans $\R^r$ and
 let  $B_0 =(b_1,\ldots, b_r) \subseteq \R^r$ be an arbitrary basis for $\R^r$ that is not necessarily contained in $\BB(A)$. Let
 $X=(A,B_0)$ and let
\begin{align}
 \ex : \{ I \subseteq A : I \text{ linearly independent} \} \to \BB(X) 
\end{align}
 be the function that maps an independent set in $A$ to its greedy extension.
 This means that 
 given an independent set $I\subseteq A$, the vectors $b_1,\ldots, b_r$ 
 are added successively to $I$
 unless the resulting set would be linearly dependent.
  
 Then we define the  set of \emph{external bases} $\BB_+(A,B_0)$ 
 and the set of \emph{internal bases} $\BB_-(A)$  by
\begin{align}
 \BB_+(A,B_0) &:= \{ B \in \BB(X) : B = \ex(I) \text{ for some } I\subseteq A \text{ independent} \} \\
  \text{ and }\; \BB_-(A) &:= \{ B\in \BB(A) : B \text{ contains no internally active elements} \}.
\end{align}
\end{Definition}

 The \emph{lattice of flats} $\Lcal(A)$ of the matroid $(A,\BB(A))$ 
  is the set $ \{ C \subseteq A : \clos (C)  = C  \}$ ordered by inclusion. 
 An \emph{upper set} $J\subseteq \Lcal(A)$ is an upward closed set, \ie $C_1 \subseteq C_2$ and $C_1 \in J$ implies $C_2\in J$.

\begin{Definition}[Semi-internal and semi-external bases \cite{holtz-ron-xu-2012}]
\label{Definition:SemiInternalExternalBases}
 We use the same terminology as in Definition~\ref{Definition:InternalExternalBases}.
 In addition, we fix an upper set $J$ in the 
lattice of flats $\Lcal(A)$ of the matroid $(A,\BB(A))$.
 For the semi-internal space, we fix an independent set $I_0\subseteq A$ 
 whose elements are maximal in $A$.

 Then we define the  set of \emph{semi-external bases} $\BB_+(A,B_0,J)$ 
 and the set of \emph{semi-internal bases} $\BB_-(A,I_0)$  by
\begin{align}
\begin{split}
 \BB_+(A,B_0,J) &:= \{ B \in \BB(X) : B=\ex(I) \text{ for some $I\subseteq A$ independent} \\
 & \qquad \qquad
\text{and } 
\clos(I) \in J \} \;\text{ and} 
\end{split} \\
 \BB_-(A,I_0) &:= \{ B\in \BB(A) : B\cap I_0 \text{ contains no internally active elements} \}.
 \nonumber
\end{align}
\end{Definition}

\begin{Definition}[Generalised external bases \cite{li-ron-2013}]
\label{Definition:LiRonBases}
Let $A=(a_1,\ldots, a_n)\subseteq \R^r$ be a list of vectors. %
 Let $\kappa : \Lcal(A) \to \{0,1,2,\ldots\}$ be an non-decreasing function,
 \ie $C_1\subseteq C_2$ implies $\kappa(C_1)\le \kappa(C_2)$.

 Let $X= (A, Y)$, where $Y= (y_1, y_2,\ldots, y_{\kappa(A) + r})$ is a list of  generic vectors, \ie
 if $y_i$ is in the span of $Z\subseteq X\setminus y_i$, then $\spa(Z) = \spa (X)$.

Then we define
\begin{align}
  \BB_\kappa(A,Y) := 
     \{ B \in \BB(X) : B \cap Y \subseteq (y_1,\ldots, y_{\kappa(\clos(A\cap B)) + \abs{B \cap Y}}) \}. 
\end{align}
\end{Definition}
\begin{Remark}
 The spaces $\Pcal(X,\BB')$ and $\Dcal(X,\BB')$ are equal to
 \begin{itemize}
  \item  the external spaces $\Pcal_+(X)$ and $\Dcal_+(X)$ in \cite{holtz-ron-2011} if $\BB'$ is the set of external bases;
  \item  the semi-external spaces $\Pcal_+(X,J)$ and $\Dcal_+(X,J)$ in \cite{holtz-ron-xu-2012} 
	  if $\BB'$ is the set of semi-external bases;
  \item the generalised external spaces $\Pcal_\kappa(X)$ and $\Dcal_\kappa(X)$ in \cite{li-ron-2013}
     if   $\BB'$ is the set of generalised external bases.
    For $\Pcal_\kappa(X)$, we need to assume in addition that $\kappa$ is incremental, \ie
        for two flats $C_1\subseteq C_2$, $\kappa(C_2)-\kappa(C_1) \le \dim(C_2) - \dim(C_1)$.
 \end{itemize}
 Furthermore, the space $\Dcal(X,\BB')$ is equal to the
 (semi-)internal space $\Dcal_-(X)$ resp.\ $\Dcal_-(X,I_0)$ in \cite{holtz-ron-2011,holtz-ron-xu-2012} 
 if $\BB'$ is the set of (semi-)internal bases.
\end{Remark}
\begin{Remark}
 The (semi-)internal spaces  $\Pcal(X,\BB_-(X))$ and $\Pcal(X,\BB_-(X,I_0))$ are in general different from
  the spaces $\Pcal_-(X)$ and $\Pcal_-(X,I_0)$  
 in \cite{holtz-ron-2011, holtz-ron-xu-2012} (cf.~\cite[Proposition 2]{ardila-postnikov-errata-2015}), but they have the
   same Hilbert series.  
\end{Remark}
\begin{Remark}
  The theorems about duality of certain $\Pcal$-spaces and $\Dcal$-spaces in
  \cite{holtz-ron-2011, holtz-ron-xu-2012, li-ron-2013} are all
 special cases of
 Theorem~\ref{Theorem:MainTheoremGeneral}.
This is a consequence of
Lemma~\ref{Prop:ZonotopalBasesSFP} below.
\end{Remark}
\begin{Lemma}
 \label{Prop:ZonotopalBasesSFP}
 The sets of bases defined in Definitions~\ref{Definition:InternalExternalBases},
 \ref{Definition:SemiInternalExternalBases}, and
 \ref{Definition:LiRonBases} all have the forward exchange property.
\end{Lemma}

\begin{proof}
  We use the following notation throughout the proof:  $B=(b_1,\ldots, b_r)$ is a basis and 
  $x \in (S_i^B \setminus S_{i-1}^B) \cap ( X\setminus E(B))$ for some $i\in [r]$.
 In addition, $B':=(B\cup x) \setminus b_i$.
 Since $x$ is not externally active, $x \le b_i$ holds. We may even assume $x<b_i$ because if equality occurs, nothing needs to be shown.

Internal and external bases are special cases of semi-internal and semi-external 
 bases so we do not consider them separately.

Let us start with the semi-external bases. Let $B\in \BB_+(A,B_0,J)$, \ie 
 B is the greedy extension of an independent set $I\subseteq A$.
 Recall that $x<b_i$.
 Hence, $x\in A$ because if $x$ was in $B_0$, the greedy extension of $I$ would contain $x$ instead of $b_i$.
  Now one can easily check that $B'\cap A$ is independent and that $\ex(B'\cap A)=B'$. This is equivalent 
 to $B'\in \BB_+(A,B_0,J)$.
  
 Now we  consider the semi-internal bases. %
 Let $B\in \BB_-(A,I_0)$. %
 By construction,
 the fundamental cocircuits of $B$ and $b_i$ resp.\ $B'$ and $x$ are equal.
 As $x< b_i$, inactivity of $b_i$ implies inactivity of $x$.
 Hence, $B'\in \BB_-(A,I_0)$.

 Last, let us consider the generalised external bases. Let $B\in \BB_{\kappa}(A,Y)$.
 If $b_i\in A$, then $\kappa ( \clos(A\cap B)) + \abs{B\cap Y}  = \kappa( \clos( A\cap B') ) + \abs{B'\cap Y}$. This 
 implies $B'\in \BB_{\kappa}(A,Y)$.
 If $b_i\in Y$,
 then $i=r$ must hold because the vectors in $Y$ are generic and we are supposing $b_i\neq x$.
 Replacing $b_r$ by $x$ reduces
 the index of the maximal element in $Y$ that is permitted in $B'$ by at most one since $\kappa $
 is non-decreasing on $\Lcal(A)$. Since we remove the maximal element of the basis $B$, this causes no problems.
\end{proof}
\begin{Remark}
 Various $\Pcal$-spaces
 without a dual $\Dcal$-space
 have been studied by several other authors \eg
 \cite{ardila-postnikov-2009,berget-2010, lenz-hzpi-2012,orlik-terao-1994, wagner-1999}.
 The approach in \cite{ardila-postnikov-2009,lenz-hzpi-2012} is slightly different from ours. 
 The authors of these two papers 
  start with a list of vectors $X$ and an integer $k$.
 In our construction, 
   it is eventually necessary to add additional elements to the list $X$ 
  in order to obtain arbitrarily large $\Pcal$-spaces. 
 In their construction, it is sufficient to let the integer parameter $k$ grow while keeping the list $X$ fixed. 
 
 Their construction of a basis for the $\Pcal$-space takes into account the internal activity of the 
 bases in $\BB(X)$. Every element of $\BB(X)$ with internally active elements may define 
 multiple elements of the basis for the $\Pcal$-space.
\end{Remark}
\begin{Example}
 Example~\ref{Example:PnotInverseSystem} fits into the framework of \cite{li-ron-2013} resp.\  
 Definition~\ref{Definition:LiRonBases}
 if we choose $A=(x_1)$, $Y=(x_2,\ldots, x_N)$ and $\kappa$ as follows: 
  $\kappa(\spa(e_1)):=N-2$ and on all other one-dimensional flats $C$, $\kappa(C):=0$.
 In this case, our space $\Dcal(X,\BB')$ is the same as the space $\Dcal_\kappa$ in \cite{li-ron-2013}. 
 The $\Pcal$-spaces are different because $\kappa$ is not incremental.
\end{Example}

 \newcommand{\MR}[1]{} 

 
\bibliographystyle{amsplain}
\providecommand{\bysame}{\leavevmode\hbox to3em{\hrulefill}\thinspace}
\providecommand{\MR}{\relax\ifhmode\unskip\space\fi MR }
\providecommand{\MRhref}[2]{%
  \href{http://www.ams.org/mathscinet-getitem?mr=#1}{#2}
}
\providecommand{\href}[2]{#2}

\end{document}